\documentclass[review,onefignum,onetabnum]{siamart190516}
\usepackage{tikz}
\usepackage{subfigure}
\usepackage{graphicx}

\newcommand{\R}{\mathbb{R}}
\newcommand{\Z}{\mathbb{Z}}

\newcommand{\WR}{\widehat{\mathcal{W}}_R}
\newcommand{\px}{\partial_x}

\newcommand{\py}{\partial_y}
\newcommand{\pt}{\partial_t}
\newcommand{\dx}{\Delta x}
\newcommand{\dt}{\Delta t}
\newcommand{\ww}{\widetilde{w}}

\newcommand{\Fr}{\text{Fr}}
\newcommand{\be}{\begin{equation}}
\newcommand{\ee}{\end{equation}}

\usepackage{amsfonts}
\usepackage{graphicx}
\usepackage{epstopdf}
\usepackage{algorithmic}
\ifpdf
  \DeclareGraphicsExtensions{.eps,.pdf,.png,.jpg}
\else
  \DeclareGraphicsExtensions{.eps}
\fi

\newsiamremark{remark}{Remark}
\newsiamremark{hypothesis}{Hypothesis}
\crefname{hypothesis}{Hypothesis}{Hypotheses}
\newsiamthm{claim}{Claim}

\headers{A fully well-balanced scheme for SWE with Coriolis force}{V. Desveaux, A. Masset}

\title{A fully well-balanced scheme for shallow water equations with Coriolis force}

\author{Vivien Desveaux\and Alice Masset}

\usepackage{amsopn}

\ifpdf
\hypersetup{
  pdftitle={A fully well-balanced scheme for the SWE with Coriolis force},
  pdfauthor={V. Desveaux, A. Masset}
}
\fi

\begin{document}

\maketitle

\begin{abstract}
	The present work is devoted to the derivation of a fully well-balanced and positive-preserving numerical scheme for the shallow water equations with Coriolis force. The first main issue consists in preserving all the steady states, including the geostrophic equilibrium. Our strategy relies on a Godunov-type scheme with suitable source term and steady state discretisations. The second challenge lies in improving the order of the scheme while preserving the fully well-balanced property. A modification of the classical methods is required since no conservative reconstruction can preserve all the steady states in the case of rotating shallow water equations. A steady state detector is used to overcome this matter. Some numerical experiments are presented to show the relevance and the accuracy of both first-order and second-order schemes.
\end{abstract}





\section{Introduction}\label{sec:intro}

In the present work we consider the one-dimensional shallow water system with transverse velocity and Coriolis force. This system is also known as 1D rotating shallow-water equations (RSW) and is given by
\be \label{eq:RSW1D}
\begin{cases} 
	 \pt h + \px (hu) = 0, \\ 
	 \pt (hu) + \px\left(hu^2 + \frac{gh^2}{2}\right) = fhv - gh\px z,\\ 
	 \pt (hv) + \px (huv) = -fhu,
\end{cases}
\ee
where $h(x,t)$ denotes the fluid height, $u(x,t)$ and $v(x,t)$ are the two components of the horizontal velocity, $z(x)$ designates the topography and is a given function, $g$ is the constant gravitational acceleration and $f$ the Coriolis parameter. 
This system can be written under the more compact form $\px w + \px f(w) =  s(w,z)$ with
$$w =\begin{pmatrix} h \\ hu \\ hv \end{pmatrix}, \quad
f(w) = \begin{pmatrix}
hu \\ hu^2+\frac{gh^2}{2} \\ huv
\end{pmatrix}, 
$$ 
and $s(w,z) = s_{cor}(w) + s_{topo}(w)\px z$ where we have set
$$ 
s_{cor}(w) = \begin{pmatrix}
0 \\ fhv \\ -fhu
\end{pmatrix}
\text{ and }
s_{topo}(w) = \begin{pmatrix}
0 \\ -gh \\ 0
\end{pmatrix}.$$ 
The first source term  is related to the Coriolis force and the second one to the topography. 
The vector $w$ must belong to the convex set of admissible states $$\Omega = \{ w=(h,hu,hv)^T \in \R^3 ;  h>0 \}.$$
This 1D system can be obtained from the two-dimensional RSW equations,
\be \label{eq:RSW2D}
\begin{cases}
	 \pt h + \px (hu) + \py (hv) = 0,\\
	 \pt (hu) + \px\left(hu^2 + \frac{gh^2}{2}\right) + \py(huv) = fhv - gh\px z,\\
	 \pt (hv) + \px (huv) + \py\left(hv^2+\frac{gh^2}{2}\right) = -fhu - gh\py z,
\end{cases}
\ee
in which the variations in the $y$ direction are neglected. 

The RSW system takes into account the force due to the Earth's rotation through the Coriolis term and can therefore model large-scale oceanic or atmospheric fluid flows. One of the remarkable behaviour of geophysical flows is the geostrophic equilibrium, that received a great attention in the literature these last years, see \cite{28BouchutSommerZeitlin2004,36bookZeitlinBouchut2007,19LukacovaNoelleKraft2007,30AudusseKleinNguyen2011,25Lahaye2014,17Audusse2017,20GouzienLahayeZetilinDubos2017,29ChertockDudzinski2018} for instance. Most oceanic and atmospheric circulations are perturbations of the geostrophic equilibrium, which express the balance between the Coriolis force and the horizontal pressure force, as follows in 2D
$$g\nabla (h+z) = f\begin{pmatrix} v \\ -u \end{pmatrix}.$$
In 1D, the geostrophic equilibrium writes 
\be \label{eq:geostrophic steady state}
\begin{cases}
	 u = 0,\\
	 g\px (h+z) = fv,
\end{cases} 
\ee
which is a steady solution of \eqref{eq:RSW1D} with no tangential velocity. Let us notice that in 1D, all the steady solutions of \eqref{eq:RSW1D} with no tangential velocity are described by the geostrophic equilibrium \eqref{eq:geostrophic steady state}. With a zero velocity $v$, we recover the lake at rest solution of the classical shallow-water model.


From a numerical point of view, it is well-known since the pioneer works \cite{37bermudez1994,38greenberg1996,39gosse2000,40jin2001}, that numerical schemes should capture accurately the steady solutions. In the few last decades, a large literature was devoted to design such well-balanced schemes able to preserve steady solution at rest in different contexts. For the classical shallow-water equations, we can mention the hydrostatic reconstruction method proposed in \cite{16Audusse2004HYDRO} and numerous other works using various methods, including \cite{liang2009numerical,fjordholm2011well,fernandez2008consistent}.
Concerning the RSW system, some authors have developed numerical schemes which preserves exactly the geostrophic equilibrium \eqref{eq:geostrophic steady state}, for instance in \cite{28BouchutSommerZeitlin2004,29ChertockDudzinski2018,19LukacovaNoelleKraft2007,22LiuChertock2019}.

More recently, some numerical schemes able to preserve all the steady states, including the moving ones, were derived. Let us emphasize that it is in general a very challenging task to derive such fully well-balanced schemes. The first attempt was in \cite{43castro2007}, where a scheme that captures all the steady states of the shallow-water equations with topography was presented. However, this scheme was not able to preserve the positivity of the water height. In \cite{bouchut2010subsonic}, the authors obtain a scheme that preserves all the sonic steady states. The first fully well-balanced and positive preserving scheme was derived by Berthon-Chalons \cite{32BerthonChalons2016}. Later, fully well-balanced schemes were also derived for the shallow-water equations with both topography and friction in \cite{47Berthon_Dansac_Manning_Friction} and for the blood flow equations in \cite{46Berthon_Blood_flow}.

For the 1D RSW equations, the steady solutions are described by
\be \label{eq:steadystates}
\begin{cases}
	 \px (hu) = 0, \\
	 \px\left(hu^2 + \frac{gh^2}{2} \right) = fhv - gh\px z, \\
	 (hu) \px v = -fhu.
\end{cases}
\ee
%
Up to our knowledge, no fully well-balanced scheme was proposed for the 1D RSW equations. In this system, there is an additional difficulty due to the complex structure of the steady states. Indeed, let us notice that the  steady solutions with nonzero tangential velocity satisfy
\be \label{eq:moving_steady_states}
\begin{cases}
	 \px (hu) = 0, \\
	 \px\left(\frac{u^2}{2} + g(h+z)\right) = fv, \\
	 \px v = -f.
\end{cases}
\ee
Thus the steady solutions with no tangential velocity described by \eqref{eq:geostrophic steady state} cannot be obtained by setting $u=0$ in \eqref{eq:moving_steady_states}. It leads to two different families of steady states. This is a discrepancy with the standard shallow-water model, where the lake at rest can be obtained by setting $u=0$ in the moving steady states equations.
The first aim of this paper is therefore to derive a fully well-balanced and positive preserving scheme for the one-dimensional RSW equations.

Another issue arises with the 1D RSW equations when we try to increase the order of precision, while preserving the well-balanced property. For other systems with source terms, well-balanced second-order extensions exist. The reader is referred for instance to \cite{15BouchutLIVRE,34DansacBerthonClainFoucher2016} for the shallow-water system with topography, \cite{47Berthon_Dansac_Manning_Friction} for the shallow-water system with both topography and friction and \cite{46Berthon_Blood_flow} for the blood flow equations. In all these extensions, the main ingredient lies in a reconstruction procedure that preserves the discrete steady states. Unfortunately, such a procedure is not possible in the case of the 1D RSW, once again due to the complex structure of the steady solutions. 

In \cite{47Berthon_Dansac_Manning_Friction} and \cite{46Berthon_Blood_flow}, a discrete steady state detection procedure is performed. The purpose is to modify the limitation procedure in order to recover the well-balanced first-order scheme near steady states and keep the high-order scheme far from steady state. We propose to adapt this technique for the 1D RSW equations. However, in order for this method to work in this context, we must complement this technique with some new manipulations of the space steps.

The paper is organized as follows. In \cref{sec:GTS}, we start by recalling some general notions about Godunov-type schemes and we choose the discretisation of the continuous steady solutions the scheme will have to preserve. Next, \cref{sec:approximate_Riemann_solver} is devoted to the derivation of an approximate Riemann solver that lead to a fully well-balanced and positive preserving scheme, as stated in \cref{thm:first-order_scheme}. In \cref{sec:secondorder}, we recall the principle of the classical second-order MUSCL extension and we explain why it cannot give a fully well-balanced scheme for the RSW system. Therefore, we present a new strategy based on a discrete steady state detection to recover this property. We also check this modification does not create non-positive fluid height values. In \cref{sec:numerical_results}, we show some numerical examples that illustrates the fully well-balanced property and the accuracy of both first-order and second-order schemes. Finally, we give some concluding remarks in \cref{sec:conclusions}.

All along this paper, for any quantity $X$ which has a left value $X_L$ and a right value $X_R$, we will use the following notations
$$[X]=X_R-X_L,\qquad \overline{X}=\frac{X_L+X_R}{2}.$$

\section{Godunov-type scheme} \label{sec:GTS}
The numerical scheme we will derive to approximate system \eqref{eq:RSW1D} is a Godunov-type scheme. In this section, we recall the framework of this family of finite volume schemes.

\subsection{Principle}\label{sec:principle}
In the following, we consider a space discretisation made of cells $K_i = (x_{i-1/2},x_{i+1/2})$, with constant length $\dx$. The center of the cell $K_i$ is denoted by $x_i$.
The topography is discretized by
$$z_i=\frac{1}{\Delta x}\int_{K_i}z(x)dx.$$
At time $t^n$, we assume known an approximation of the solution of \eqref{eq:RSW1D} constant on each cell, $$w_{\Delta x}(x,t^n) = w_i^n, \text{ if } x \in K_i.$$
In order to simplify the notations, we set $\ww = (w, z)$, which belongs to the set
$$\widetilde{\Omega}=\{ \ww=(h,hu,hv,z)^T \in \R^3 ;  h>0 \}.$$
Since $z$ does not depend on time, we have $\ww_i^n=(w_i^n,z_i)$.

We aim to update this approximation at time $t^{n+1} = t^n+\dt$, with a step $\dt$ chosen according to a CFL condition. 
Godunov-type schemes are mainly based on Riemann problems, which are Cauchy problems for system \eqref{eq:RSW1D} with an initial data of the form 
\be \label{eq:riemann_pb}
\ww(x,0) = \begin{cases} 
 \ww_L & \text{ if } x<0, \\
 \ww_R & \text{ if } x>0. 
\end{cases} \ee
We denote by $\mathcal{W}_R(\frac{x}{t},\ww_L,\ww_R)$ the exact solution of \eqref{eq:RSW1D}--\eqref{eq:riemann_pb}.
This exact solution is usually very difficult to compute. Therefore, we prefer to use an approximate Riemann solver $\WR\left(\frac{x}{t},\ww_L,\ww_R\right)$ instead. According to \cite{31HLL1983}, the approximate Riemann solver has to satisfy the following consistency property:
\begin{equation*}
 \frac{1}{\dx} \int_{-\frac{\dx}{2}}^{\frac{\dx}{2}} \WR\left(\frac{x}{\dt},\ww_L,\ww_R\right) dx 
 = \frac{1}{\dx}\int_{-\frac{\dx}{2}}^{\frac{\dx}{2}} \mathcal{W}_R\left(\frac{x}{\dt},\ww_L,\ww_R\right) dx.
\end{equation*}
The average of the exact Riemann solution can be computed and the previous condition is equivalent to
\begin{multline} \label{eq:strong_consistency}
 \frac{1}{\dx} \int_{-\frac{\dx}{2}}^{\frac{\dx}{2}} \WR\left(\frac{x}{\dt},\ww_L,\ww_R\right) dx =  \frac{w_L+w_R}{2} - \frac{\dt}{\dx}(f(w_R)-f(w_L)) \\
+  \frac{1}{\dx} \int_0^{\dt} \int_{-\frac{\dx}{2}}^{\frac{\dx}{2}} s\left(\mathcal{W}_R\left(\frac{x}{t},\ww_L,\ww_R\right), z(x) \right) dxdt.
\end{multline}
In the absence of source term, we can enforce this equality to ensure the consistency of the approximate Riemann solver. However, it is not always possible to compute exactly the average of the source term. Therefore, it is usual to use a relevant approximation (see for instance \cite{15BouchutLIVRE,32BerthonChalons2016,12Ripa2016})
$$ S(\ww_L,\ww_R)\approx\frac{1}{\dt} \int_0^{\dt} \int_{-\frac{\dx}{2}}^{\frac{\dx}{2}} s\left(\mathcal{W}_R\left(\frac{x}{t},\ww_L,\ww_R\right), z(x)\right) dxdt.$$
This numerical source term should be consistent with the continuous source term $s$ in the following sense.

\begin{definition} \label{def:source_term_consistency}
 The numerical source term $S$ is consistent with the continuous source term $s(\ww) = s_{cor}(w) + s_{topo}(w)\px z$ if it satisfies
  \begin{equation} \label{eq:general_source_consistency}
 S((w,z_L),(w,z_R)) = s_{cor}(w) \dx + s_{topo}(w) [z].
 \end{equation}
\end{definition}

Provided a consistent numerical source term, the approximate Riemann solver can only satisfy a weaker version of \eqref{eq:strong_consistency}. It leads to the definition of a weakly consistent approximate Riemann solver.

\begin{definition}\label{def:weak_consistency}
	The approximate Riemann solver $\WR$ is weakly consistent if there exists a consistent numerical source term $S$ such that	
	\be \label{eq:weak_consistency}
	\frac{1}{\dx}\int_{-\frac{\dx}{2}}^{\frac{\dx}{2}} \WR\left(\frac{x}{\dt},\ww_L,\ww_R\right) dx = \frac{w_L+w_R}{2} - \frac{\dt}{\dx}(f(w_R)-f(w_L)) + \frac \dt \dx S(\ww_L,\ww_R).
	\ee	
\end{definition}

The following section will be devoted to derive a weakly consistent approximate Riemann solver. For now, we show how we can obtain a numerical scheme from an approximate Riemann solver. A Godunov-type scheme is built in two steps:

\begin{itemize}
 \item firstly, we consider the juxtaposition of approximate Riemann solvers at each interface $x_{i+1/2}
$, 
$$
w_{\Delta x}(x,t^n+t) = \WR\left(\frac{x-x_{i+1/2}}{t},\ww^n_{i},\ww_{i+1}^n \right), \text{ if }x \in (x_{i},x_{i+1});
$$

\item secondly, the update at time $t^{n+1}$ is obtained by averaging the previous function on each cell 
$$ w_{i+1}^n = \frac{1}{\dx}\int_{x_{i-1/2}}^{x_{i+1/2}} w_{\Delta x}(x,t^n+\dt)dx, 
$$
or equivalently
\begin{equation} \label{eq:scheme_update}
w_{i}^{n+1} = \frac{1}{\dx} \int_0^{\frac{\dx}{2}} \WR\left(\frac{x}{\dt},\ww_{i-1}^n,\ww_i^n\right)dx + \frac{1}{\dx} \int_{-\frac{\dx}{2}}^0 \WR\left(\frac{x}{\dt},\ww_i^n,\ww_{i+1}^n \right)dx.
\end{equation}
\end{itemize}

In order to prevent the approximate Riemann solvers to interact between each other, we must enforce the CFL restriction
\begin{equation} \label{eq:CFL_order1}
	\frac{\dt}{\dx} \max_{i \in \Z}| \lambda^\pm(\ww_i^n,\ww_{i+1}^n) | \leq \frac{1}{2},
\end{equation} 
where $\lambda^\pm(\ww_L,\ww_R)$ denotes both the maximum and minimum speed of the waves that appear in $\WR(\frac{x}{t},\ww_L,\ww_R)$.
Under this condition, we can write the Godunov-type scheme as a finite volume scheme (see for instance \cite{32BerthonChalons2016})
\begin{equation}  \label{eq:godunov_type_scheme}
w_i^{n+1} = w_i^n - \frac{\dt}{\dx}(F_{i+1/2}^n - F_{i-1/2}^n) + \frac{\dt}{2 \dx} (S_{i+1/2}^n + S_{i-1/2}^n),
\end{equation}
with $F_{i+1/2}^n = F(\ww_i^n,\ww_{i+1}^n)$ and $S_{i+1/2}^n = S(\ww_i^n,\ww_{i+1}^n)$, where the numerical flux is given by
\begin{multline} \label{eq:numerical flux general def}
	F(\ww_L,\ww_R) = \frac{f(w_L)+f(w_R)}{2} - \frac{\dx}{4\dt}(w_R-w_L) \\
	+ \frac{1}{2\dt} \left( \int_0^{\frac{\dx}{2}}\WR\left(\frac{x}{\dt},\ww_L,\ww_R \right)dx - \int_{-\frac{\dx}{2}}^0 \WR\left(\frac{x}{\dt},\ww_L,\ww_R \right) dx \right),
\end{multline}
and the numerical source term $S(\ww_L,\ww_R)$ is the same as introduced in \cref{def:weak_consistency}.

At this point, the only property the scheme has to satisfy is the weak consistency of the approximate Riemann solver. We now list some other properties the scheme should satisfy.

\subsection{Numerical scheme properties}\label{sec:numerical_scheme_properties}

We present two important features of numerical schemes in this context: robustness and well-balancing. Godunov-type schemes have the advantage of inheriting these properties from the approximate Riemann solver $\WR$. First, we study the preservation of fluid height positivity.

\begin{lemma}\label{lem:robust}
 If the approximate Riemann solver $\WR$ satisfies the robustness condition
 \be \label{eq:RS_robustness}
 \forall (\ww_L,\ww_R) \in \widetilde{\Omega}^2,\, \forall \xi \in \R,\, \WR\left(\xi,\ww_L,\ww_R\right) \in \Omega,
 \ee
 then under the CFL condition \eqref{eq:CFL_order1}, the Godunov-type scheme \eqref{eq:godunov_type_scheme} preserves the positivity of the fluid height:
$$\forall i \in \Z, h_i^n>0  \Rightarrow  \forall i \in \Z, h_{i}^{n+1}>0.$$
\end{lemma}
\begin{proof}
Assuming $w_i^n \in \Omega $ for all $i \in \Z,$ the state $w_i^{n+1}$ defined by \eqref{eq:scheme_update} appears to be an average of elements that belong to the convex set $\Omega$.
\end{proof}

Similarly, the Godunov-type scheme is well-balanced as soon as the approximate Riemann solver is. To be more specific, we have to introduce the notion of local steady state.

\begin{definition}\label{def:local steady state}
A couple of states $(\ww_L,\ww_R)$ defines a local steady state for the system \eqref{eq:RSW1D} if it satisfies
\be \label{eq:steady states discretisation}
\begin{cases}
	 h_Ru_R=h_Lu_L=q, \\
	 \left[ \frac{u^2}{2} + g(h+z) \right] = \dx f\overline{v}, \\
	 q[v] = - \dx fq,
\end{cases}
\ee
or equivalently if the local steady state indicator 
\be \label{eq:steady state indicator epsLR}\mathcal{E}(\ww_L,\ww_R,\dx) = \sqrt{\Big\vert \left[ hu \right] \Big\vert^2 + \Bigg\vert \left[\frac{u^2}{2}+g (h+z) \right] - \dx f\overline{v} \Bigg\vert^2 + \Big\vert \overline{hu} ( [v] + f \Delta x ) \Big\vert^2}, 
\ee
is equal to zero.
\end{definition}

All along this paper, we write $\mathcal{E}_{LR}$ rather than $\mathcal{E}(\ww_L,\ww_R,\dx)$ if no ambiguity is possible.

Let us notice that \eqref{eq:steady states discretisation} is actually a discretization of the equations \eqref{eq:steadystates} that define the continuous steady states. Other choices of discretization could be possible, especially in the choice of the mean value $\overline{v}$.

The definition of a well-balanced Riemann solver follows.

\begin{definition}\label{def:ARS WB}
	An approximate Riemann solver is said well-balanced if
	$$\WR\left(\frac{x}{t},\ww_L,\ww_R\right) = 
	\begin{cases}
	 w_L & \text{ if } x>0, \\
	 w_R & \text{ if } x<0, 
	\end{cases} $$ as soon as $(\ww_L,\ww_R)$ is a local steady state.
\end{definition}

Similarly, we define a discrete steady state and a well-balanced scheme.

\begin{definition}\label{def:approx at steady state}    $~$
	\begin{enumerate}
		\item 	A sequence $(\ww_i^n)_{i \in \Z}$ defines a discrete steady state if the couples $(\ww_i,\ww_{i+1})$ are local steady states for all $i \in \Z$.
		\item A numerical scheme is said well-balanced if for any discrete steady state $(\ww_i^n)_{i \in \Z}$, we have $$ w_i^{n+1} = w_i^n, \; \forall i \in \Z.$$ 
	\end{enumerate}

\end{definition}

Now we prove that the well-balanced property of the approximate Riemann solver extends to the numerical scheme.

\begin{lemma}\label{lem:WB}
If the approximate Riemann solver $\WR$ is well-balanced, then the associated Godunov-type scheme \eqref{eq:godunov_type_scheme} is well-balanced too.
\end{lemma}
\begin{proof}
	We consider a discrete steady state $(\ww_i^n)_{i \in \Z}$. Since the approximate Riemann solver is well-balanced, we get from \eqref{eq:scheme_update} that
$$w_i^{n+1} = \frac{1}{\dx} \int_0^{\frac{\dx}{2}} w_i^n dx + \frac{1}{\dx} \int_{-\frac{\dx}{2}}^0 w_i^n dx = w_i^n, \ \forall i \in \Z,$$
and the proof is complete.
\end{proof}

To summarize, the approximate Riemann solver that we will derive in the next section has to satisfy the following properties:
\begin{itemize}
\item \label{item:consistency with exact solver} the weak consistency condition \eqref{eq:weak_consistency},
\item \label{item:robust} the robustness condition \eqref{eq:RS_robustness},
\item \label{item:fully WB} the fully well-balanced property given by \cref{def:ARS WB}.
\end{itemize}

\section{Approximate Riemann solver} \label{sec:approximate_Riemann_solver} Here, we propose an approximate Riemann solver for the system \eqref{eq:RSW1D} that satisfies the three previous properties. All along its construction, we will carefully choose the used relations for this purpose.  We adapt to the RSW system the strategy proposed in \cite{34DansacBerthonClainFoucher2016,47Berthon_Dansac_Manning_Friction,46Berthon_Blood_flow} for different systems.

\subsection{Source term discretisation}\label{sec:numerical source term}

The aim of this section is to propose a numerical source term $S(\ww_L,\ww_R)=(0,S^{hu}(\ww_L,\ww_R),S^{hv}(\ww_L,\ww_R))^T$ which is consistent with the continuous source term $s$ in the sense of \cref{def:source_term_consistency}. Moreover this choice of a numerical source term has to be coherent with the required well-balanced property.

To this end, we start by considering a Riemann data $(\ww_L,\ww_R)$ which is a local steady state according to \cref{def:local steady state}. Since we want the approximate Riemann solver to be both weakly consistent and well-balanced, the condition \eqref{eq:weak_consistency} enforces
\be \label{eq:source_term_flux}
S(\ww_L,\ww_R) = f(w_R) - f(w_L),
\ee
or equivalently 
\begin{align*}
& S^{hu}(\ww_L,\ww_R) = h_Ru_R^2 + \frac{gh_R^2}{2} - h_Lu_L^2 - \frac{gh_L^2}{2},\\
& S^{hv}(\ww_L,\ww_R) = h_Ru_Rv_R - h_Lu_Lv_L.
\end{align*}
Based on the chosen \cref{def:local steady state} of the local steady states, these relations can be written
\begin{align}
\label{eq:Shu at steady state, alpha_LR}& S^{hu}(\ww_L,\ww_R) = \left( g\overline{h} - \frac{q^2}{h_Lh_R}\right) [h],\\
\label{eq:Shv at steady state}& S^{hv}(\ww_L,\ww_R) = -\dx fq.
\end{align}

The expression \eqref{eq:Shu at steady state, alpha_LR} cannot be used to define the numerical source term in the general case since it would not be consistent in the sense of \cref{def:source_term_consistency}.
Hence, we continue to develop this expression for a local steady state.
%
First, from the second equality of \eqref{eq:steady states discretisation}, we get
\be 
\frac{q^2}{2h_R^2} + g(h_R+z_R) - \frac{q^2}{2h_L^2} -g(h_L+z_L) = \dx f \overline{v}, 
\ee
which leads to 
\be \label{eq:equilibrium} 
[h]\left(1-\frac{q^2 \overline{h}}{gh_L^2h_R^2}\right) = \dx f \overline{v}/g - (z_R - z_L).
\ee
It follows
\be \label{eq:relation at steady states between hR-hL and zR-zL}
	 [h] = \frac{\dx f\overline{v}/g - [z]}{1-\Fr},
\ee
where $\Fr = \frac{\overline{h}|u_Lu_R|}{gh_Lh_R}$ is a discrete Froude number. 
Injecting this relation into \eqref{eq:Shu at steady state, alpha_LR}, we get
\begin{equation}\label{eq:Shu at steady state0}
 S^{hu}(\ww_L,\ww_R) = \dx f\overline{h}\overline{v}-g\overline{h}[z] + \frac{g\Fr [h]^2}{4\overline{h}} \frac{ (\dx f\overline{v}/g - [z])}{(1-\Fr)}.
\end{equation}
%
We inject one more time \eqref{eq:relation at steady states between hR-hL and zR-zL} in the above equality to obtain a more convenient expression 
\be \label{eq:Shu at steady state, zR-zL}
	 S^{hu}(\ww_L,\ww_R) = \dx f\overline{h}\overline{v}-g\overline{h} [z]+ \frac{g\Fr [h]}{4\overline{h}} \frac{ (\dx f\overline{v}/g - [z])^2}{(1-\Fr)^2}.
\ee

This expression is \emph{a prior}i not well-defined when $\Fr = 1$. However, combining \eqref{eq:relation at steady states between hR-hL and zR-zL} and \eqref{eq:Shu at steady state0} leads to rewrite $S^{hu}(\ww_L,\ww_R)$ under the form
$$S^{hu}(\ww_L,\ww_R) = g\overline{h}[h](1-\Fr) + \frac{g}{4\overline{h}}\Fr [h]^3.$$
Therefore $ S^{hu}(\ww_L,\ww_R)$ admits the following limit when $\Fr$ goes to $1$
\be \label{eq:Shu at steady state + Fr = 1}
\lim_{\Fr \to 1} S^{hu}(\ww_L,\ww_R) =  \frac{g}{4\overline{h}} [h]^3.
\ee

At this point, \eqref{eq:Shu at steady state, zR-zL} and \eqref{eq:Shv at steady state} are suitable definitions for the numerical source terms $S^{hu}$ and $S^{hv}$ when a local steady state is considered.
However, let us point out that the limit \eqref{eq:Shu at steady state + Fr = 1} is only valid for a local steady state. Therefore, the right-hand side of \eqref{eq:Shu at steady state, zR-zL} is not well-defined when $\varepsilon_{LR}\neq 0$ and $\Fr=1$. To deal with this issue, we add a nonnegative term $\mathcal{E}_{LR}$ to the denominator as follows
\begin{equation} \label{eq:Shu generalised not well-defined}
 S^{hu}(\ww_L,\ww_R) = \dx f\overline{h}\overline{v}-g\overline{h} [z]+ \frac{g \Fr [h]}{4\overline{h}} \frac{ (\dx f\overline{v}/g - [z])^2}{(1- \Fr)^2+\mathcal{E}_{LR}}.
\end{equation}
Indeed, the denominator in \eqref{eq:Shu generalised not well-defined} can only vanish when $\varepsilon_{LR}=0$, which means the Riemann data $(\ww_L,\ww_R)$ is a local steady state. But then the source term can be defined by the limit \eqref{eq:Shu at steady state + Fr = 1} as mentioned before.

To generalise \eqref{eq:Shv at steady state} away from local steady states,  we need to define a general discharge $\widetilde{q}$ which coincides with $q$ as soon as a local steady state is considered or as soon as $w_L=w_R=w$. There are several possible definitions, $\widetilde{q} = \overline{hu}$ for instance.


 We finally obtain the following definitions for the numerical source terms
\begin{multline} \label{eq:Shu_general_definition}
 S^{hu}(\ww_L,\ww_R) = \\
 \begin{cases}
  \displaystyle \dx f\overline{h}\overline{v}-g\overline{h}[z]+\frac{g \Fr [h]}{4\overline{h}} \frac{ (\dx f\overline{v}/g - [z])^2}{(1-\Fr)^2 + \mathcal{E}_{LR} } & \text{if } \Fr \neq 1 \text{ or } \mathcal{E}_{LR} \neq 0,\\[1.5em]
  \displaystyle \frac{g}{4\overline{h}}[h]^3 & \text{if } \Fr = 1 \text{ and } \mathcal{E}_{LR} = 0.
 \end{cases}
\end{multline}
\be \label{eq:Shv_general_definition}
S^{hv}(\ww_L,\ww_R) = -\dx f \widetilde{q}.
\ee

To conclude, we prove these numerical source terms are consistent.
\begin{lemma}\label{lem:numerical source term definition}
The numerical source term $S(\ww_L,\ww_R)=(0,S^{hu}(\ww_L,\ww_R),S^{hv}(\ww_L,\ww_R))^T$ defined by \eqref{eq:Shu_general_definition} and \eqref{eq:Shv_general_definition} is consistent in the sense of \cref{def:source_term_consistency}.
%
\end{lemma}

\begin{proof}
	The consistency is immediate for $S^{hv}$, as for $S^{hu}$ in the case $\Fr \neq 1$ or $\mathcal{E}_{LR} \neq 0$. In the case $\Fr = 1$ and $\mathcal{E}_{LR} = 0$, let us notice that according to \eqref{eq:equilibrium}, we have $\dx f\overline{v}/g = [z]$. Therefore the source term $S^{hu}$ can be written under the form
	$$ S^{hu}(\ww_L,\ww_R) = \dx f\overline{h}\overline{v}-g\overline{h}[z] +  \frac{g[h]^3}{4\overline{h}},$$ 
	and the consistency follows. 
\end{proof}

%

\subsection{Approximate Riemann solver} \label{sec:ARS}

The numerical source term being well-defined, we now turn to build a weakly consistent approximate Riemann solver which is fully well-balanced and preserves the positivity of the fluid height.

Let us notice that the well-balanced property of the approximate Riemann solver strongly depends on the choice that was made in  \cref{def:local steady state} to discretise the steady states. However, the following procedure stands for any discretisation of the steady states. This is not the case for the source term discretisation which was done in the previous section and should be adapted to the steady state discretisation.

We consider a Riemann data $(\ww_L,\ww_R)\in\widetilde{\Omega}^2$. We choose to build an approximate Riemann solver $\WR$ with four constant states separated by three discontinuities with respective speed $\lambda_L < 0$, $\lambda_0 = 0$ and $\lambda_R>0$, as described in \cref{fig:Approximate Riemann solver}. This approximate Riemann solver writes
\begin{equation}\label{eq:ARS}
\WR\left(\frac{x}{t},\ww_L,\ww_R\right) = 
\begin{cases}
	 w_L   & \text{ if } \frac{x}{t}< \lambda_L, \\
	 w_L^\star &\text{ if }  \lambda_L<\frac{x}{t}<0, \\
	 w_R^\star &\text{ if }  0<\frac{x}{t}<\lambda_R, \\
	 w_R   &\text{ if }  \frac{x}{t}>\lambda_R.
\end{cases}
\end{equation}
This leads to two intermediate states $w_L^\star$ and $w_R^\star$ and thus six unknowns. We are searching for as many equations as unknowns.
\begin{figure}[h!] \label{fig:Approximate Riemann solver}
\centering
\begin{tikzpicture}[scale=0.5]
	\draw[-] (-6,0) -- (6,0);
    \draw[-] (-3,5.5) -- (0,0);
    \draw[-] (0,5.5) -- (0,0);
    \draw[-] (4.5,5.5) -- (0,0);
                
	\draw[black] (-4.5,2) node {$w_L$};
    \draw[black] (-1.,4) node {$w_L^\star$};
    \draw[black] (1.5,4) node {$w_R^\star$};
    \draw[black] (4.5,2) node {$w_R$};
                   
	\draw[black] (-3.,5) node [above left] {$\lambda_L$};
    \draw[black] (0,5.5) node [above] {$0$};
    \draw[black] (4.5,5.) node [above right] {$\lambda_R$};
\end{tikzpicture} 
\caption{Approximate Riemann solver $\WR$}
\end{figure}
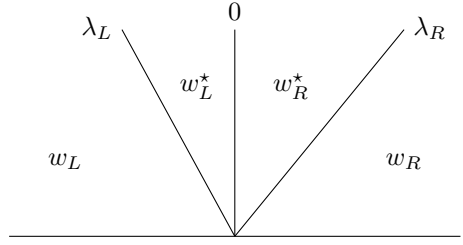

In order to simplify the subsequent notations, we introduce the intermediate state of the HLL approximate Riemann solver (see \cite{31HLL1983}) 
\be \label{eq:def_HLL}
w^{HLL} = \frac{\lambda_R w_R - \lambda_L w_L}{\lambda_R - \lambda_L} - \frac{f(w_R) - f(w_L)}{\lambda_R - \lambda_L}.
\ee
Let us notice that its first component can be written as
$$h^{HLL} = \frac{u_L-\lambda_L}{\lambda_R-\lambda_L}h_L + \frac{\lambda_R-u_R}{\lambda_R-\lambda_L}h_R.$$
As a consequence, as soon as the speeds $\lambda_L$ and $\lambda_R$ satisfy
\begin{equation} \label{eq:lambda}
\lambda_L<u_L\qquad\text{and}\qquad \lambda_R>u_R,
\end{equation}
we have $h^{HLL}>0$.

First, the weak consistency condition  \eqref{eq:weak_consistency} writes after a standard computation
\begin{align}
	\label{eq:Riemann solver relations consistency 1} & \lambda_R h_R^\star - \lambda_L h_L^\star = (\lambda_R-\lambda_L) h^{HLL}, \\
	\label{eq:Riemann solver relations consistency 2} & \lambda_R h_R^\star u_R^\star - \lambda_L h_L^\star u_L^\star = (\lambda_R - \lambda_L) (hu)^{HLL} + S^{hu}(\ww_L,\ww_R), \\
	\label{eq:Riemann solver relations consistency 3} & \lambda_R h_R^\star v_R^\star - \lambda_L h_L^\star v_L^\star  = (\lambda_R - \lambda_L) (hv)^{HLL} + S^{hv}(\ww_L,\ww_R).
\end{align}
It provides three relations that will ensure the weak consistency of the approximate Riemann solver.
The three missing relations will come from the fully well-balanced constraint. In other words, we have to choose three additional relations such that the solution of the system formed by these relations and equations \eqref{eq:Riemann solver relations consistency 1}, \eqref{eq:Riemann solver relations consistency 2} and \eqref{eq:Riemann solver relations consistency 3} satisfies
$$ w_R^\star = w_R \quad \text{and} \quad w_L^\star = w_L$$
as soon as $(\ww_L,\ww_R)$ is a local steady state.
We will deal with each variable separately.

First, for the variable $hu$, the simplest choice is to enforce the relation
\begin{equation} \label{eq:link steady state source term intermediate 1}
 h_L^\star u_L^\star = h_R^\star u_R^\star = q^\star.
\end{equation}
The system \eqref{eq:Riemann solver relations consistency 2}--\eqref{eq:link steady state source term intermediate 1} can be solved immediately to obtain the intermediate discharge
\begin{equation} \label{eq:qstar}
 q^\star = (hu)^{HLL} + \frac{S^{hu}(\ww_L,\ww_R)}{\lambda_R - \lambda_L}.
\end{equation}

Concerning the variable $h$, let us notice that when $(\ww_L,\ww_R)$ is a local steady state, we have according to \eqref{eq:Shu at steady state, alpha_LR}
$$\alpha_{LR}(h_R-h_L)=S^{hu}(\ww_L,\ww_R),$$
where $\alpha_{LR} = g\overline{h} - |u_Lu_R|$. A simple choice for the additional equation would be
$$\alpha_{LR}(h_R^\star - h_L^\star) = S^{hu}(\ww_L,\ww_R).$$
Together with equation \eqref{eq:Riemann solver relations consistency 1}, this leads to a simple linear system. However this system does not admit a unique solution when $\alpha_{LR}$ vanishes. We suggest the following modification
\begin{equation*}
 (\alpha_{LR}^2+\mathcal{E}_{LR})(h_R^\star - h_L^\star) = \alpha_{LR}S^{hu}(\ww_L,\ww_R).
\end{equation*}
The coefficient $\alpha_{LR}^2+\mathcal{E}_{LR}$ can still vanish when $\mathcal{E}_{LR}=0$. To get rid of this problem, we introduce the following quantity
$$\Delta_{LR}^h =
\begin{cases} \displaystyle\frac{\alpha_{LR}S^{hu}(\ww_L,\ww_R)}{\alpha_{LR}^2+ \mathcal{E}_{LR}} & \text{if } \mathcal{E}_{LR}\neq 0, \\
 h_R-h_L & \text{if } \mathcal{E}_{LR}=0,
\end{cases} $$
and we choose the following additional equation
\be h_R^\star - h_L^\star = \Delta_{LR}^h. \label{eq:additional_eq_h}
\ee
Solving the system \eqref{eq:Riemann solver relations consistency 1}--\eqref{eq:additional_eq_h}, we obtain
$$ h_L^\star = h^{HLL} - \frac{\lambda_R}{\lambda_R-\lambda_L} \Delta_{LR}^h, $$
$$ h_R^\star = h^{HLL} - \frac{\lambda_L}{\lambda_R-\lambda_L}\Delta_{LR}^h. $$

Nothing ensures these intermediate fluid heights to be positive. To address this issue, we adapt the cut-off procedure suggested in \cite{53Audusse_Chalons_Ung,34DansacBerthonClainFoucher2016,46Berthon_Blood_flow}. Let us introduce the threshold
\be\delta=\min(\varepsilon, h_L,h_R,h^{HLL}) \label{eq:delta},
\ee
where $\varepsilon>0$ is a small parameter.
If $h_L^\star<\delta$, we set $h_L^\star=\delta$, and $h_R^\star$ is modified according to \eqref{eq:Riemann solver relations consistency 1}. In this case, we have
$$h_R^\star = \left(1-\frac{\lambda_L}{\lambda_R}\right)h^{HLL}+\frac{\lambda_L}{\lambda_R} h_L^\star\ge \delta,$$
so both intermediate water heights $h_L^\star$ and $h_R^\star$ are positive. We proceed similarly if $h_R^\star<\delta$. Taking into account this procedure, the intermediate fluid heights write
\be \label{eq:hLstar}
h_L^\star = \min\left(\max\left(h^{HLL}-\frac{\lambda_R}{\lambda_R-\lambda_L}\Delta_{LR}^h,\delta\right), \left(1-\frac{\lambda_R}{\lambda_L}\right)h^{HLL}+\frac{\lambda_R}{\lambda_L}\delta\right),
\ee
\be \label{eq:hRstar}
h_R^\star = \min\left(\max\left(h^{HLL}-\frac{\lambda_L}{\lambda_R-\lambda_L}\Delta_{LR}^h,\delta\right), \left(1-\frac{\lambda_L}{\lambda_R}\right)h^{HLL}+\frac{\lambda_L}{\lambda_R}\delta\right).
\ee

Finally, we proceed similarly in order to derive the additional relations for the variable $hv$. We introduce the quantity
$$\Delta_{LR}^{v} = 
\begin{cases}
 \displaystyle\frac{\widetilde{q} \; S^{hv}(\ww_L,\ww_R)}{{\widetilde{q}}^2+\mathcal{E}_{LR}} & \text{if } \mathcal{E}_{LR}\neq 0, \\
 v_R - v_L & \text{if } \mathcal{E}_{LR}=0,
\end{cases} $$
and we enforce the following additional equation
\be \label{eq:additional_eq_v}
v_R^\star - v_L^\star = \Delta_{LR}^v.
\ee
The system \eqref{eq:Riemann solver relations consistency 3}--\eqref{eq:additional_eq_v} then leads to
\be
v_L^\star = \frac{(hv)^{HLL}}{h^{HLL}} + \frac{1}{(\lambda_R-\lambda_L)h^{HLL}}\left(S^{hv}(\ww_L,\ww_R)-\lambda_R h_R^\star\Delta_{LR}^v\right), \label{eq:vLstar}
\ee
\be
v_R^\star = \frac{(hv)^{HLL}}{h^{HLL}} + \frac{1}{(\lambda_R-\lambda_L)h^{HLL}}\left(S^{hv}(\ww_L,\ww_R)-\lambda_L h_L^\star\Delta_{LR}^v\right). \label{eq:vRstar}
\ee

The approximate Riemann solver is then completely defined by relations \eqref{eq:link steady state source term intermediate 1}, \eqref{eq:qstar}, \eqref{eq:hLstar}, \eqref{eq:hRstar}, \eqref{eq:vLstar} and \eqref{eq:vRstar}. Let us notice that it is automatically weakly consistent by the choice of the three first equations \eqref{eq:Riemann solver relations consistency 1}, \eqref{eq:Riemann solver relations consistency 2} and \eqref{eq:Riemann solver relations consistency 3}. The cut-off procedure does not alter the weak consistency, since equation \eqref{eq:Riemann solver relations consistency 1} is still enforced when it is applied. Moreover, thanks to the cut-off procedure, both intermediate fluid heights are positive as soon as the speeds $\lambda_L$ and $\lambda_R$ satisfy \eqref{eq:lambda}.

We now prove the approximate Riemann solver is also well-balanced.

\begin{lemma} \label{lem:ARS_WB}
 The approximate Riemann solver $\WR$ is well-balanced.
\end{lemma}

\begin{proof}
 Let us consider a local steady state $(\ww_L,\ww_R)$. To prove the result, we only need to show that $w_L^\star=w_L$ and $w_R^\star=w_R$. First, let us assume the cut-off procedure does not apply.
 Since $(w_L^\star,w_R^\star)$ is defined as the unique solution of the system of equations \eqref{eq:Riemann solver relations consistency 1}, \eqref{eq:Riemann solver relations consistency 2}, \eqref{eq:Riemann solver relations consistency 3}, \eqref{eq:link steady state source term intermediate 1}, \eqref{eq:additional_eq_h}, \eqref{eq:additional_eq_v}, it is sufficient to prove that $(w_L,w_R)$ is solution of this system.
 
 Since we have $\mathcal{E}_{LR}=0$, it is immediate that $(w_L,w_R)$ is solution of \eqref{eq:link steady state source term intermediate 1}, \eqref{eq:additional_eq_h}, \eqref{eq:additional_eq_v}. The approximate solver being weakly consistent and $(\ww_L,\ww_R)$ being a local steady state, equation \eqref{eq:weak_consistency} enforces
 $$ S(\ww_L,\ww_R)=f(w_R)-f(w_L),$$
 so the intermediate state of the HLL solver defined by \eqref{eq:def_HLL} satisfies
 $$(\lambda_R-\lambda_L)w^{HLL}= \lambda_R w_R - \lambda_L w_L - S(\ww_L,\ww_R).$$
 As a consequence, equations \eqref{eq:Riemann solver relations consistency 1}, \eqref{eq:Riemann solver relations consistency 2} and \eqref{eq:Riemann solver relations consistency 3} rewrite
 \begin{align*}
  & \lambda_R h_R^\star - \lambda_L h_L^\star = \lambda_R h_R - \lambda_L h_L, \\
  & \lambda_R h_R^\star u_R^\star - \lambda_L h_L^\star u_L^\star = \lambda_R h_R u_R - \lambda_L h_L u_L, \\
  & \lambda_R h_R^\star v_R^\star - \lambda_L h_L^\star v_L^\star  = \lambda_R h_R v_R - \lambda_L h_L v_L.
 \end{align*}
 We deduce $(w_L,w_R)$ is a solution of these equations and thus $w_L^\star=w_L$ and $w_R^\star=w_R$.
 
 Finally, we state that the cut-off procedure cannot apply in this case. Indeed, thanks to the definition \eqref{eq:delta}, the intermediate fluid heights computed before the cut-off procedure satisfy $h_L^\star=h_L\ge\delta$ and $h_R^\star=h_R\ge\delta$.
\end{proof}

The approximate Riemann solver $\WR$ thus satisfies all the required properties.

\subsection{The final scheme}\label{sec:the_final_scheme}

We summarize in this section the full scheme and its properties.

\begin{theorem}\label{thm:first-order_scheme}
	The approximate Riemann solver \eqref{eq:ARS} where the intermediate states are given by \eqref{eq:link steady state source term intermediate 1}, \eqref{eq:qstar}, \eqref{eq:hLstar}, \eqref{eq:hRstar}, \eqref{eq:vLstar} and \eqref{eq:vRstar} leads to a Godunov-type scheme that can be written under the form \eqref{eq:godunov_type_scheme}. The numerical flux
	$$F(\ww_L,\ww_R)= \left( F^{h}(\ww_L,\ww_R), F^{hu}(\ww_L,\ww_R), F^{hv}(\ww_L,\ww_R) \right)^T,$$
	is given by 
	\begin{align*}
		& F^h(\ww_L,\ww_R)  = \overline{hu} + \frac{\lambda_R}{2}(h_R^\star - h_R) + \frac{\lambda_L}{2}(h_L^\star-h_L),\\
		& F^{hu}(\ww_L,\ww_R) =\overline{hu^2+\frac{gh^2}{2}} + \frac{\lambda_R}{2}(h_R^\star u_R^\star - h_Ru_R) +\frac{\lambda_L}{2}(h_L^\star u_L^\star - h_L u_L), \\
		& F^{hv}(\ww_L,\ww_R) = \overline{huv} + \frac{\lambda_R}{2}(h_R^\star v_R^\star - h_Rv_R) + \frac{\lambda_L}{2} (h_L^\star v_L^\star-h_Lv_L),
	\end{align*}
	and the numerical source term
	$$S(\ww_L,\ww_R)= \left( 0, S^{hu}(\ww_L,\ww_R), S^{hv}(\ww_L,\ww_R) \right)^T$$
	is defined by \eqref{eq:Shu_general_definition} and \eqref{eq:Shv_general_definition}. \\
	Under the CFL restriction \eqref{eq:CFL_order1} and if the speeds $\lambda_L$ and $\lambda_R$ are chosen accordingly to \eqref{eq:lambda}, this scheme is fully well-balanced and preserves the positivity of $h$.
\end{theorem}

\begin{proof}
 The expression of the numerical flux is obtained from a straightforward computation in \eqref{eq:numerical flux general def}.

 Assume $(w_L,w_R)$ are in $\Omega$. Since $h^{HLL} >0 $, the cut-off procedure ensures $h_L^\star > 0$ and $h_R^\star > 0$. Thus the variable $h$ remains positive in the approximate Riemann solver. According to \cref{lem:robust}, the scheme preserves the positivity of $h$.
 
 The well-balanced property of the scheme is a direct consequence of \cref{lem:WB,lem:ARS_WB}.
\end{proof}

\section{Second-order scheme} \label{sec:secondorder}
In this section, we propose to improve the scheme precision using the MUSCL method. Our goal is to build a second-order scheme in space that preserves the good properties of the first-order one, namely the positivity of $h$ and the well-balanced property. The second-order in time is obtained with the usual Runge-Kutta method. We do not describe it here, but the reader can refer to \cite{shu1988efficient,gottlieb1998total,08Ordre2Berthon2001}.

We start by a description of the standard MUSCL method, and we explain why it is not adapted to get the fully well-balanced property for the RSW system. Indeed, no conservative reconstruction can preserve the complex structure of all the steady states defined by \eqref{eq:steadystates}. We explain in \cref{sec:Fully well-balanced recovering} how to recover the fully well-balanced property by adapting the ideas proposed in \cite{47Berthon_Dansac_Manning_Friction} and \cite{46Berthon_Blood_flow} to our generalised MUSCL scheme.

%
%
%
%
%

Up to this point, for the sake of conciseness, we neither mentioned explicitly the dependence on $\dx$ in the numerical fluxes and the source terms, nor in the definition of local steady states. However in the following, we will consider half-cells, which will impose to make appear these dependencies, in particular to determine if the scheme is fully well-balanced. It will also be useful to consider the $\dx$ that appears in the approximate Riemann solver $\WR$ and the $\dx$ that appears in the numerical scheme definition \eqref{eq:godunov_type_scheme} as two separated parameters. For $d > 0$ the approximate Riemann solver is given by \begin{equation*}
\WR\left(\frac{x}{t},\ww_L,\ww_R ,d\right) = 
 \begin{cases}
 w_L  &  \text{ if } \frac{x}{t}< \lambda_L, \\
 w_L^\star(d) & \text{ if }  \lambda_L<\frac{x}{t}<0, \\
 w_R^\star(d) & \text{ if }  0<\frac{x}{t}<\lambda_R, \\
 w_R   & \text{ if }  \frac{x}{t}>\lambda_R.
\end{cases}
\end{equation*}
According to \cref{sec:principle}, the resulting Godunov-type scheme writes \begin{multline}\label{eq:godunov_type_scheme_d}
	w_i^{n+1} = w_i^n - \frac{\dt}{\dx}\left( F \left(\widetilde{w}_i^n,\widetilde{w}_{i+1}^n,d \right) - F \left(\widetilde{w}_{i-1}^n,\widetilde{w}_{i}^n, d \right) \right) \\
	+ \frac{\dt}{2\dx} \left( S \left(\widetilde{w}_{i-1}^n,\widetilde{w}_i^n,d \right) + S \left(\widetilde{w}_i^n,\widetilde{w}_{i+1}^n,d \right) \right),
\end{multline}
provided the CFL condition \eqref{eq:CFL_order1} is satisfied. Notice that for $d = \dx$, we recover the fully well-balanced scheme derived in \cref{sec:approximate_Riemann_solver}. Moreover, we establish the following lemma, that will be useful for the forthcoming proof.

\begin{lemma} \label{lem:schema_d_robust}
	Under the CFL condition \eqref{eq:CFL_order1} and if the approximate Riemann solver speed waves $\lambda_L$ and $\lambda_R$ satisfy \eqref{eq:lambda}, then the Godunov-type scheme \eqref{eq:godunov_type_scheme_d} preserves the positivity of $h$, for all $d > 0$.
\end{lemma}

\begin{proof}
	Independently of the parameter $d$, the cut-off procedure leads to positive intermediate states $h_L^\star$ and $h_R^\star$ according to definition \eqref{eq:hLstar}-\eqref{eq:hRstar}, since the speed waves $\lambda_L$ and $\lambda_R$ satisfy the condition \eqref{eq:lambda}. Then, we apply the \cref{lem:robust} to conclude the scheme \eqref{eq:godunov_type_scheme_d} preserves the positivity of $h$ for all $d > 0$.
\end{proof}


\subsection{Standard MUSCL method}\label{sec:standard_MUSCL_method}

The main idea of the MUSCL method is to reach second-order by considering a linear reconstruction of the solution on each cell, instead of a constant one. We recall here the standard reconstruction procedure.

Starting from a piecewise approximation at time $t^n$, 
$$\widetilde{w}_{\Delta x}(x,t^n) = \ww_i^n \text{ if } x \in K_i,$$
we reconstruct a linear approximation on each cell 
$$\widehat{w}_{\Delta x}(x,t^n) = \sigma_i^n(\widetilde{w}) (x-x_i) + \ww_i^n, \text{ if } x \in K_i, $$
where $\sigma_i^n(\widetilde{w})$ is a slope vector to determine. Let us emphasize that this procedure includes the topography.
The reconstructed states correspond to the value of $\widehat{w}_{\Delta x}$ at the interfaces of each cell and read as 
\begin{equation}\label{eq:reconstruction}
\widetilde{w}_i^{n,\pm} = \widetilde{w}_i^n \pm \frac \dx 2 \sigma_i^n(\widetilde{w}),
\end{equation}

To avoid spurious oscillations, it is well-known that a limitation procedure must be applied to the slopes.
In this paper, we consider the minmod limiter function defined by
$$\mathrm{minmod}(\sigma_L,\sigma_R) =
\begin{cases}
\min(\sigma_L,\sigma_R) & \text{if } \sigma_L>0 \text{ and } \sigma_R > 0, \\
\max(\sigma_L,\sigma_R) & \text{if } \sigma_L<0 \text{ and } \sigma_R < 0, \\
0 & \text{otherwise}.
\end{cases}
$$
Then the slope vector is defined by $\sigma_i^n (\widetilde{w}) = \mathrm{minmod} \left( \frac{\widetilde{w}_i^n - \widetilde{w}_{i-1}^n}{\dx}, \frac{\widetilde{w}_{i+1}^n - \widetilde{w}_i^n}{\dx} \right)$. Other limiters can be considered, see \cite{14Toro2009,53Leveque2003} for instance. As usual, we enforce an additional limitation procedure on the first component of the slope vector $\sigma_i^n(\widetilde{w})$ in order to ensure that the fluid heights $h_i^{n,\pm}$ remain positive.

The standard MUSCL extension is obtained as follows.
For a first-order scheme under the form \eqref{eq:godunov_type_scheme},
the second-order scheme is defined by
\begin{multline} \label{eq:second-order scheme not WB}
	w_i^{n+1} = w_i^n - \frac{\dt}{\dx} (F(\widetilde{w}_i^{n,+},\widetilde{w}_{i+1}^{n,-},d) - F(\widetilde{w}_{i-1}^{n,+},\widetilde{w}_i^{n,-},d) ) \\
	+ \frac{\dt}{2\dx} ( S(\widetilde{w}_{i-1}^{n,+},\widetilde{w}_i^{n,-},d) + S_{c}(\widetilde{w}_i^{n,-},\widetilde{w}_i^{n,+},d) + S(\widetilde{w}_i^{n,+},\widetilde{w}_{i+1}^{n,-},d) ),
\end{multline}
where $S_{c}(\widetilde{w}_L,\widetilde{w}_R,d)$ is a centered source term, added to take into account the source term when its jumps at interfaces are small (see \cite{15BouchutLIVRE}). Several possibilities exist to define it.
In the present work, it will be determined naturally by considering the MUSCL scheme \eqref{eq:second-order scheme not WB} as a convex combination between two first-order schemes applied on the reconstructed states and on half-cells, as described in \cref{fig:MUSCL}. More precisely, we define
\begin{multline*}
	w_i^{n+1,+} = w_i^{n,+} - \frac{\dt}{\dx/2} \left( F \left(\widetilde{w}_i^{n,+},\widetilde{w}_{i+1}^{n,-},\frac \dx 2 \right) - F \left(\widetilde{w}_{i}^{n,-},\widetilde{w}_i^{n,+}, \frac \dx 2 \right) \right) \\
	+ \frac{\dt}{\dx} \left( S \left(\widetilde{w}_{i}^{n,-},\widetilde{w}_i^{n,+},\frac \dx 2\right) + S \left(\widetilde{w}_i^{n,+},\widetilde{w}_{i+1}^{n,-},\frac \dx 2 \right) \right),
\end{multline*}
and 
\begin{multline*}
	w_i^{n+1,-} = w_i^{n,-} - \frac{\dt}{\dx/2} \left( F \left(\widetilde{w}_i^{n,-},\widetilde{w}_{i}^{n,+},\frac \dx 2 \right) - F \left(\widetilde{w}_{i-1}^{n,+},\widetilde{w}_{i}^{n,-}, \frac \dx 2 \right) \right) \\
	+ \frac{\dt}{\dx} \left( S \left(\widetilde{w}_{i-1}^{n,+},\widetilde{w}_i^{n,-},\frac \dx 2\right) + S \left(\widetilde{w}_i^{n,-},\widetilde{w}_{i}^{n,+},\frac \dx 2 \right) \right).
\end{multline*}
Taking the average of these two states, we obtain
\begin{multline}\label{eq:second-order scheme cvx}
	w_i^{n+1} = \frac{w_i^{n,-}+w_i^{n,+}}{2} - \frac{\dt}{\dx}\left( F \left(\widetilde{w}_i^{n,+},\widetilde{w}_{i+1}^{n,-},\frac \dx 2 \right) - F \left(\widetilde{w}_{i-1}^{n,+},\widetilde{w}_{i}^{n,-}, \frac \dx 2 \right) \right) \\
	+ \frac{\dt}{2\dx} \left( S \left(\widetilde{w}_{i-1}^{n,+},\widetilde{w}_i^{n,-},\frac \dx 2\right) + 2S \left(\widetilde{w}_{i}^{n,-},\widetilde{w}_i^{n,+},\frac \dx 2 \right) + S \left(\widetilde{w}_i^{n,+},\widetilde{w}_{i+1}^{n,-},\frac \dx 2 \right) \right).
\end{multline}
Assuming the reconstruction is conservative, namely $w_i^n=\frac{w_i^{n,-}+w_i^{n,+}}{2}$, we notice that the scheme \eqref{eq:second-order scheme cvx} can be written under the form \eqref{eq:second-order scheme not WB} with $d = \frac{\dx}{2}$ and by defining the centered source term as
$$ S_{c}(\widetilde{w}_i^{n,-},\widetilde{w}_i^{n,+},d) = 2 S\left(\widetilde{w}_i^{n,-},\widetilde{w}_i^{n,+},\frac{\Delta x}{2}\right).$$

An advantage of this procedure is that the MUSCL scheme \eqref{eq:second-order scheme cvx} automatically preserves the positivity of $h$ as soon as the associated first-order scheme does, up to a half CFL restriction. 

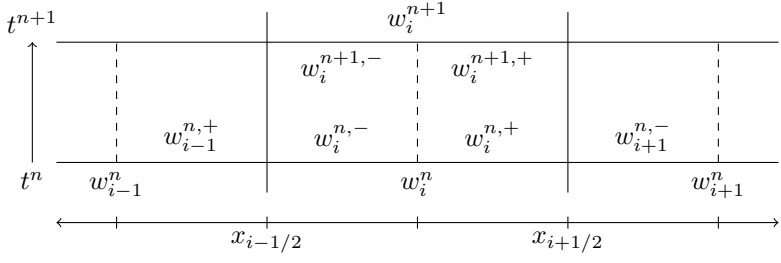
\begin{figure}
	\centering
	\begin{tikzpicture}[scale=0.8]
	
	\draw[<->] (-6,0) -- (6,0);
	\draw[->] (-6.4,1) -- (-6.4,3);
	\draw[-] (-2.5,0.5) -- (-2.5,3.5);
	\draw[-] (2.5,0.5) -- (2.5,3.5);
	\draw[dashed] (0,1) -- (0,3);
	\draw[dashed] (-5,1) -- (-5,3);
	\draw[dashed] (5,1) -- (5,3);
	
	\draw[-] (2.5,-0.1) -- (2.5,0.1);
	\draw[-] (-2.5,-0.1) -- (-2.5,0.1);
	\draw[-] (5,-0.1) -- (5,0.1);
	\draw[-] (-5,-0.1) -- (-5,0.1);
	\draw[-] (0,-0.1) -- (0,0.1);
	
	\draw[-] (-6,1) -- (6,1);
	\draw[-] (-6,3) -- (6,3);

	\draw[black] (-2.5,0) node [below] {$x_{i-1/2}$};
	\draw[black] (2.5,0) node [below] {$x_{i+1/2}$};
	
	\draw[black] (-6.4,1) node [below] {$t^n$};
	\draw[black] (-6.4,3) node [above] {$t^{n+1}$};

	\draw[black] (-5,1) node [below] {$w_{i-1}^n$};
	\draw[black] (0,1) node [below] {$w_{i}^n$};
	\draw[black] (5,1) node [below] {$w_{i+1}^n$};
	
	\draw[black] (-3.75,1) node [above] {$w_{i-1}^{n,+}$};
	\draw[black] (-1.25,1) node [above] {$w_{i}^{n,-}$};
	\draw[black] (1.25,1) node [above] {$w_{i}^{n,+}$};
	\draw[black] (3.75,1) node [above] {$w_{i+1}^{n,-}$};
	
	\draw[black] (-1.25,3) node [below] {$w_{i}^{n+1,-}$};
	\draw[black] (1.25,3) node [below] {$w_{i}^{n+1,+}$};
	
	\draw[black] (0,3) node [above] {$w_{i}^{n+1}$};
	
	\end{tikzpicture} 
	\caption{MUCL second-order scheme}
	\label{fig:MUSCL}
\end{figure}


However, the well-balanced property is not reached as easily. Indeed, in order for the MUSCL scheme \eqref{eq:second-order scheme cvx} to be well-balanced, the reconstruction would have to satisfy for any discrete steady state $(\widetilde{w}_i^n)_{i\in\Z}$, $$\mathcal{E}(\widetilde{w}_i^{n,+},\widetilde{w}_{i+1}^{n,-},\dx/2) = \mathcal{E}(\widetilde{w}_i^{n,-},\widetilde{w}_i^{n,+},\dx/2)=0, \text{ for all } i \in \Z.$$ 
Unfortunately, we cannot provide such a reconstruction in our case. Indeed, we have to reconstruct four variables, including the conservative ones $h,hu,hv$ which leaves only one free variable to reconstruct. Moreover, according to definition \eqref{eq:steady states discretisation}, the moving steady states involve three expressions among which two are not conservative quantities. Therefore, we would need to reconstruct two free variables in order to preserve steady states. 

We propose in the next section, to modify the MUSCL method and the reconstruction to get around that problem and recover a fully well-balanced second-order scheme.

\subsection{Fully well-balanced recovering}\label{sec:Fully well-balanced recovering}

We suggest a modification based on an idea introduced in \cite{47Berthon_Dansac_Manning_Friction} and \cite{46Berthon_Blood_flow}. The main principle is to consider the second-order scheme \eqref{eq:second-order scheme not WB} far from steady states and recover the first-order scheme \eqref{eq:godunov_type_scheme} near a steady state, which guarantee the scheme to be well-balanced.
The difficulty lies in the definition of being far from/close to a steady state.

For this purpose, we consider a smooth increasing function $\theta$, valued in $[0,1]$ and such that $\theta(0)=0$ and $\theta(x)\approx 1$ far from $0$. We choose the following function
$$\theta(x) = \frac{x^2}{x^2 + \dx^2}.$$
We set $\theta_i^n = \theta(\mathcal{E}_i^n)$, where $\mathcal{E}_i^n = \mathcal{E}(\widetilde{w}_{i-1}^n,\widetilde{w}_i^n,\dx) + \mathcal{E}(\widetilde{w}_i^n,\widetilde{w}_{i+1}^n,\dx)$ detects if both couples $(\widetilde{w}_{i-1}^n,\widetilde{w}_i^n)$ and $(\widetilde{w}_i^n,\widetilde{w}_{i+1}^n)$ are local steady states simultaneously.

The reconstructed states are now defined as a convex combination between the linear reconstructed states and the first-order states
\begin{equation}\label{eq:second-order_reconstruction}
	\widetilde{w}_i^{n,\pm} = (1-\theta_i^n) \widetilde{w}_i^n + \theta_i^n \left( \widetilde{w}_i^n \pm \frac \dx 2 \sigma_i^n(\widetilde{w}) \right) = \widetilde{w}_i^n \pm \theta_i^n \frac \dx 2 \sigma_i^n(\widetilde{w})
\end{equation}
This reconstruction amounts to consider an additional limitation that involves the steady state detector $\theta_i^n$.
For a discrete steady state, we have $\theta_i^n = 0$ and we recover the first-order states $\widetilde{w}_i^{n,\pm} = \widetilde{w}_i^n$. Far from steady state and for a smooth solution, a mere computation shows that $\widetilde{w}_i^{\pm} = \widetilde{w}_i^n + \frac{\dx}{2} \sigma_i^n(\widetilde{w}) + O(\dx^3)$ when $\dx$ tends to $0$, which means the perturbation added to the usual second-order reconstruction is small enough to recover the seeking order.

Next, we define the scheme as 
\begin{multline}\label{eq:second-order scheme WB}
	w_i^{n+1} = w_i^n - \frac{\dt}{\dx}\left( F \left(\widetilde{w}_i^{n,+},\widetilde{w}_{i+1}^{n,-},\dx_1 \right) - F \left(\widetilde{w}_{i-1}^{n,+},\widetilde{w}_{i}^{n,-}, \dx_1 \right) \right) \\
	+ \frac{\dt}{2\dx} \left( S \left(\widetilde{w}_{i-1}^{n,+},\widetilde{w}_i^{n,-},\dx_1\right) + 2S \left(\widetilde{w}_i^{n,-},\widetilde{w}_{i}^{n,+},\dx_2\right) + S \left(\widetilde{w}_i^{n,+},\widetilde{w}_{i+1}^{n,-},\dx_1 \right) \right),
\end{multline}
where the coefficients $\dx_1$ and $\dx_2$ have to be adapted, depending if we apply the first or second-order scheme. Far from steady states we need $\dx_1 = \dx_2 = \frac \dx 2$ in \eqref{eq:second-order scheme WB} to recover the second-order scheme \eqref{eq:second-order scheme not WB}. For a discrete steady state, the scheme reads as
\begin{multline}
	w_i^{n+1} = w_i^n - \frac{\dt}{\dx}\left( F \left(\widetilde{w}_i^n,\widetilde{w}_{i+1}^n,\dx_1 \right) - F \left(\widetilde{w}_{i-1}^n,\widetilde{w}_{i}^n, \dx_1 \right) \right) \\
	+ \frac{\dt}{2\dx} \left( S \left(\widetilde{w}_{i-1}^n,\widetilde{w}_i^n,\dx_1\right) + 2S \left(\widetilde{w}_i^n,\widetilde{w}_{i}^n,\dx_2\right) + S \left(\widetilde{w}_i^n,\widetilde{w}_{i+1}^n,\dx_1 \right) \right).
\end{multline}
We notice that $S(\widetilde{w},\widetilde{w},0) = 0$ according to the source term consistency \eqref{eq:general_source_consistency}. Therefore, we have to set $\dx_1 = \dx$ and $\dx_2 = 0$ to recover the first-order scheme \eqref{eq:godunov_type_scheme}.

In order to satisfy both these requirements, coefficients $\dx_1$ and $\dx_2$ are set as convex combinations as follows
\begin{equation} \label{eq:def_dx1_dx2}
\dx_1 = \dx \left( 1-\frac{\theta_i^n}{2} \right)\quad\text{and}\quad \dx_2 = \theta_i^n \frac \dx 2.
\end{equation}

We prove in the following theorem that the resulting second-order scheme is fully well-balanced, and that it preserves the positivity of $h$ under the classical second-order CFL restriction.

\begin{theorem}\label{thm:second-order_scheme} 

	Under the CFL condition 
	$$ \frac{\dt}{\dx} \max_{i \in \Z} \left( \vert \lambda^\pm(\widetilde{w}_i^{n,-},\widetilde{w}_i^{n,+})\vert, \vert \lambda^\pm(\widetilde{w}_i^{n,+},\widetilde{w}_{i+1}^{n,-}) \vert \right) \leq \frac{1}{4} ,$$
	and if the speeds $\lambda_L$ and $\lambda_R$ of the approximate Riemann solver satisfy the condition \eqref{eq:lambda}, then the second-order scheme \eqref{eq:second-order_reconstruction}-\eqref{eq:second-order scheme WB}-\eqref{eq:def_dx1_dx2} is fully well-balanced and preserves the positivity of $h$.
\end{theorem}

\begin{proof} First, we consider a discrete steady state $(\widetilde{w}_i^n)_{i \in \Z}$. By definition, we have $\theta_i^n = 0$ for all $i \in \Z$. Hence, the scheme \eqref{eq:second-order_reconstruction}-\eqref{eq:second-order scheme WB}-\eqref{eq:def_dx1_dx2} gives
	\begin{multline}
	w_i^{n+1} = w_i^n - \frac{\dt}{\dx}\left( F \left(\widetilde{w}_i^n,\widetilde{w}_{i+1}^n,\dx \right) - F \left(\widetilde{w}_{i-1}^n,\widetilde{w}_{i}^n, \dx \right) \right) \\
	+ \frac{\dt}{2\dx} \left( S \left(\widetilde{w}_{i-1}^n,\widetilde{w}_i^n,\dx\right) + S \left(\widetilde{w}_i^n,\widetilde{w}_{i+1}^n,\dx \right) \right),
	\end{multline}
	which is nothing but the fully well-balanced first-order scheme \eqref{eq:godunov_type_scheme}.

	Now we prove the positivity-preserving property. We assume that $h_i^n $ is positive for all $i \in \Z$. 
	The update of variable $h$ with the scheme \eqref{eq:second-order_reconstruction}-\eqref{eq:second-order scheme WB}-\eqref{eq:def_dx1_dx2} writes
	\begin{align*}
		h_i^{n+1} & = h_i^n - \frac{\dt}{\dx}\left(  F^h \left(\widetilde{w}_i^{n,+},\widetilde{w}_{i+1}^{n,-},\dx_1 \right) - F^h \left(\widetilde{w}_{i-1}^{n,+},\widetilde{w}_{i}^{n,-}, \dx_1 \right) \right)\\
		& = \frac{1}{2}\left( h_i^{n,-} - \frac{\dt}{\dx/2} \left( F^h \left(\widetilde{w}_i^{n,-},\widetilde{w}_{i}^{n,+},\dx_1 \right) - F^h \left(\widetilde{w}_{i-1}^{n,+},\widetilde{w}_{i}^{n,-},\dx_1 \right) \right) \right) \\
		& \quad + \frac{1}{2} \left( h_i^{n,+} - \frac{\dt}{\dx/2} \left( F^h \left(\widetilde{w}_i^{n,+},\widetilde{w}_{i+1}^{n,-},\dx_1 \right) - F^h \left(\widetilde{w}_{i}^{n,-},\widetilde{w}_i^{n,+}, \dx_1 \right) \right) \right).
	\end{align*}
	Then $h_i^{n+1}$ is a convex combination between first-order schemes applied on half-cells with parameter $d = \dx_1$. As proved in \cref{lem:schema_d_robust}, the first-order scheme preserves the positivity of $h$ independently of the value of the parameter $d$. Therefore, we conclude $h_i^{n+1} > 0 $ for all $i \in \Z$.
\end{proof}



\section{Numerical results}\label{sec:numerical_results}

This section is devoted to numerical experiments. For the sake of simplicity, the initial discretisation will be defined as 
$$w_i^0 = w_0(x_i).$$

Considering a continuous steady solution, the initial discretisation can satisfy exactly the \cref{def:approx at steady state} of the discrete steady states. In this case, both our first-order and second-order schemes were proved to preserve the initial condition. This will be illustrated in \cref{testcase1}.

However, it is also possible that the initial discretisation of a continuous steady state does not lead to a discrete steady state according to \cref{def:approx at steady state}. The behaviour of our numerical schemes in such a case will be investigated in \cref{testcase2}.
In order to measure how close a given discretisation $(w_i^n)_{i\in \Z}$ at time $t^n$ is to a discrete steady state, we will use the steady state distance $$ \mathcal{E}^n_{\infty,j} = \max_{1 \leq i \leq N} \mathcal{E}(w_i^n,w_{i+1}^n),$$ where $j=1$ for the first-order scheme and $j=2$ for the second-order scheme.

In \cref{testcase3}, we test the long-time convergence towards a steady state on a topography with a bump, using the same distance $\mathcal{E}_{\infty,j}^n$.

Finally, in \cref{testcase4}, we consider a particular solution constant in space, but not in time, and for which we compute the errors in time.

%

\subsection{Moving steady state} \label{testcase1}
We consider here a simple moving steady state.
As initial data, we take (see \cref{fig:MovingSS})
$$h_0(x) = \exp^{2 x}, \ u_0(x) = \exp^{-2 x} \text{ and } v_0(x) = -fx,$$ 
and the topography is given by
$$z(x) = -\frac{1}{2}f^2 x^2 - \exp^{2 x} - \frac 12  \exp^{-4 x}.$$
We compute this test on the domain $[0,1]$ with $N=200$ cells and the parameters $f=g=1$. 

The initial discretisation is a discrete steady state in the sense of definition \ref{def:approx at steady state}. Indeed the steady state distance at time $t_0 = 0$ is 
$$ \mathcal{E}_{\infty,1}^0 = \mathcal{E}_{\infty,2}^0 = 8.87 \times 10^{-16}.$$ At final time $T_{\max} =  0.5$, the steady state is still preserved by both first-order and second-order schemes, even if small computationnal errors have spread. Indeed, the computation of the steady state distance at the end of the simulations gives
$$ \mathcal{E}_{\infty,1}^{T_{\max}} = 5.19 \times 10^{-14} \quad\text{and}\quad \mathcal{E}_{\infty,2}^{T_{\max}} = 8.86 \times 10^{-15}.$$

\begin{figure}\label{fig:MovingSS}
	\centering
	\includegraphics[scale=0.75]{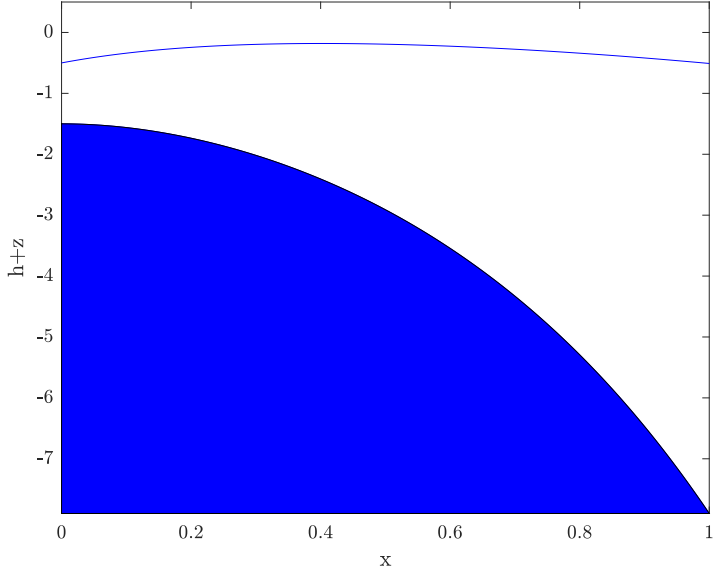}
	\caption{Initial variable h+z for the moving steady state}
\end{figure}

\subsection{Geostrophic steady state} \label{testcase2}
Next, we test the numerical schemes on another geostrophic steady state introduced in \cite{29ChertockDudzinski2018}. The computational domain is $[-5,5]$ with a flat topography ($z \equiv 0$). We set $f=10$, $g=1$, and we consider
$$h_0(x)=\frac{2}{g}-e^{-x^2}, u_0(x)=0, v_0(x)=\frac{2g}{f} xe^{-x^2},$$ 
as initial condition, which is a continuous steady state, see \cref{fig:GeosSS_Init}. The initial data discretisation is not exactly a discrete steady state, since we have for $N=200$ discretisation points $$ \mathcal{E}_{\infty,1}^0 = \mathcal{E}_{\infty,2}^0 = 4.06 \times 10^{-5}.$$
Therefore, \cref{thm:first-order_scheme,thm:second-order_scheme} does not guarantee the behaviour of the numerical schemes on this test case. However, at final time $T_{\max} = 200$ the first-order and second-order schemes  lead to the steady state distances
$$\mathcal{E}_{\infty,1}^{T_{\max}}= 1.12 \times 10^{-7} \quad \text{and}\quad \mathcal{E}_{\infty,2}^{T_{\max}} = 2.53 \times 10^{-12}.$$
Both schemes seem to converge numerically to the steady state as $t$ goes to infinty. 
Let us notice that according to \cref{sec:Fully well-balanced recovering}, the second-order scheme in space gives back the first-order scheme when the approximation is close to steady state is detected. The observed difference between the steady state distances at final time is due to the order of the time scheme. 

\begin{figure}\label{fig:GeosSS_Init}
	\centering
	\includegraphics[scale=0.75]{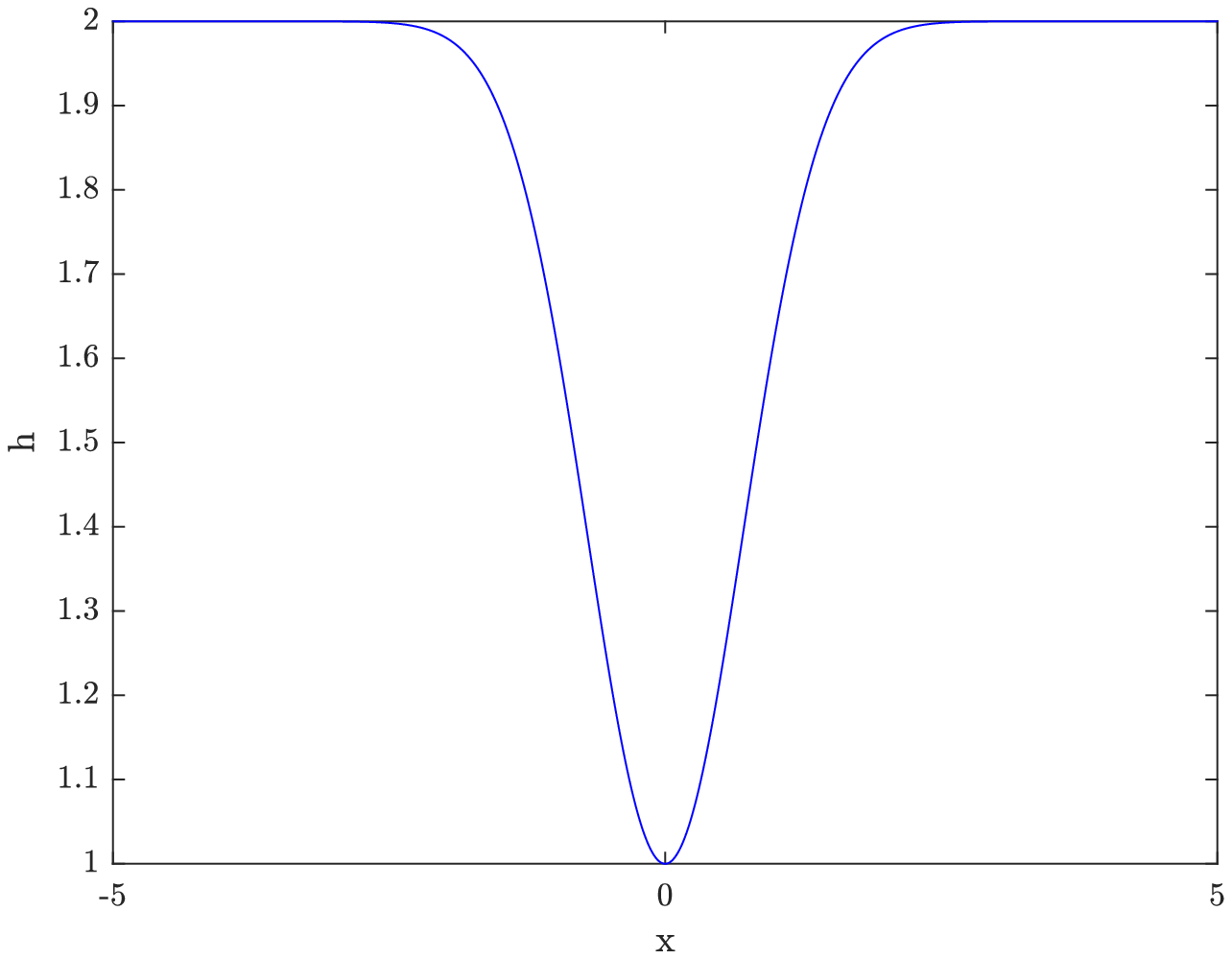}
	\includegraphics[scale=0.75]{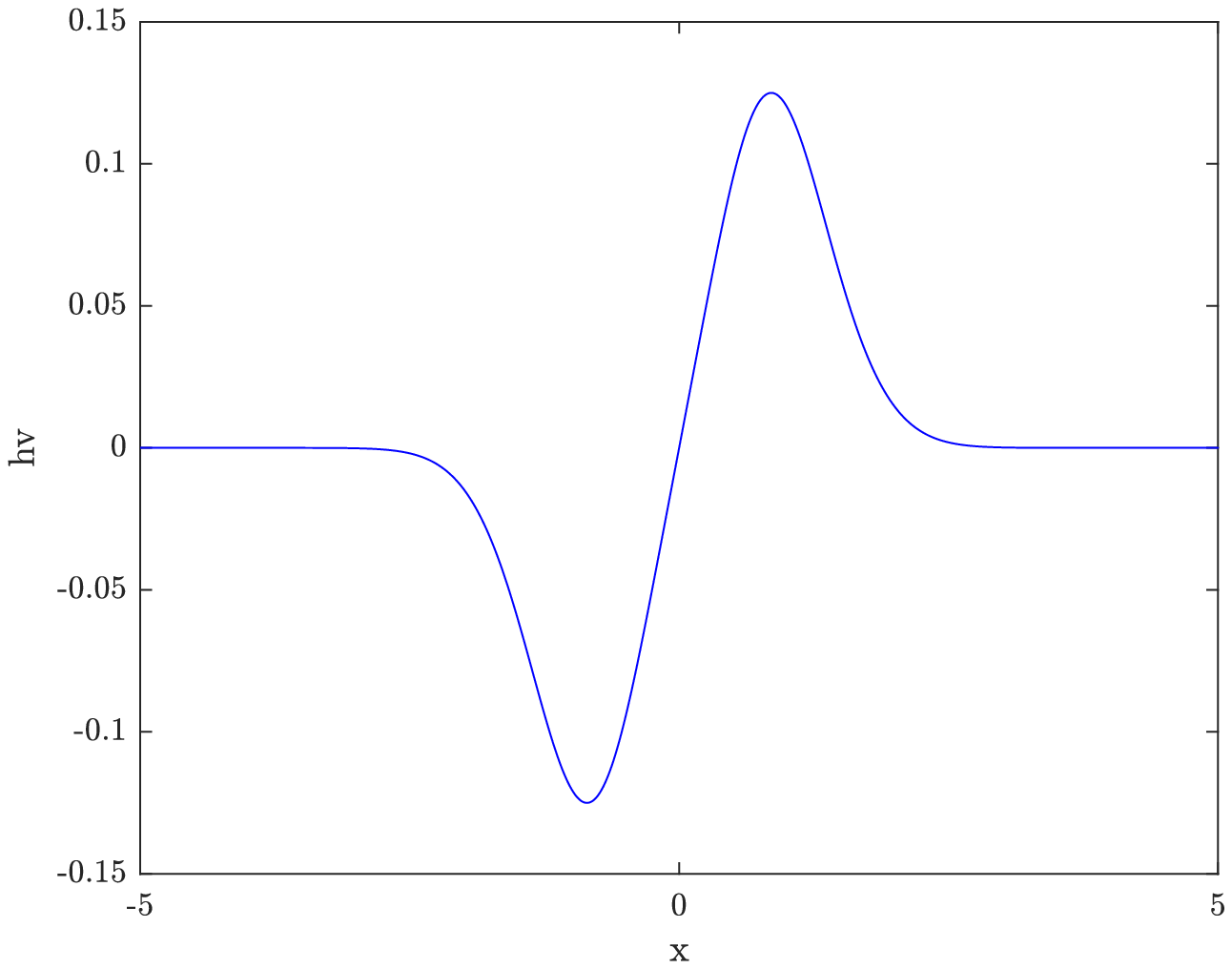}
	\caption{Initial data for the geostrophic steady state}
\end{figure}


Let us now check the convergence of both schemes when $\dx$ tends to $0$.
We define the $L_1$  discrete error in space at time $t^n $ between the exact solution $w_0$ and the numerical approximation by 
$$E^n = \dx \sum_{i=1}^N \vert w_0(x_i) - w_i^{n} \vert.$$

We present in  \cref{tab:GeoSS_error} the discrete errors for variables $h$ and $hv$ at final time $T_{\max}$ for the first-order and second-order schemes. As one can see, both schemes converge to the steady state and reach second-order accuracy in space whereas only first-order accuracy is expected for the first-order scheme.
This behaviour can be formally explained by the fact that the initial discretisation satisfies the local steady state \cref{def:local steady state} up to second-order. Indeed, a straightforward expansion shows that
%
the initial discretisation satisfies
$$ g(h_0(x+\dx) - h_0(x)) - \dx f \frac{v_0(x)+v_0(x+\dx)}{2}
= O(\dx^2).$$

\begin{table}{}\label{tab:GeoSS_error}
	\caption{$L^1$ error in space for the geostrophic steady state at time $T_{\max} = 200$ }
	\centering
	\subtable[first-order scheme]{
		\begin{tabular}{|c|c|c|c|c|}
			\hline N    &  \multicolumn{2}{c|}{h}   & \multicolumn{2}{c|}{hv} \\
			\hline 200  &  $5.25 \times 10^{-5}$ &       &  $2.11 \times 10^{-4}$ &       \\
			\hline 400  &  $1.31 \times 10^{-5}$ & 2.00  &  $5.30 \times 10^{-5}$ & 1.99  \\
			\hline 800  &  $3.30 \times 10^{-6}$ & 1.99  &  $1.38 \times 10^{-5}$ & 1.94  \\
			\hline 1600 &  $8.58 \times 10^{-7}$ & 1.94  &  $3.73 \times 10^{-6}$ & 1.88  \\
			\hline 3200 &  $2.30 \times 10^{-7}$ & 1.91  &  $1.02 \times 10^{-6}$ & 1.87  \\
			\hline 6400 &  $6.01 \times 10^{-8}$ & 1.93  &  $2.73 \times 10^{-7}$ & 1.90  \\
			\hline
	\end{tabular}}
	
	
	\subtable[second-order scheme]{
		\begin{tabular}{|c|c|c|c|c|}
			\hline N    &  \multicolumn{2}{c|}{h}   & \multicolumn{2}{c|}{hv} \\
			\hline 200  &  $5.26 \times 10^{-5}$ &       &  $2.11 \times 10^{-4}$ &       \\
			\hline 400  &  $1.31 \times 10^{-5}$ & 2.00  &  $5.27 \times 10^{-5}$ & 2.00  \\
			\hline 800  &  $3.29 \times 10^{-6}$ & 2.00  &  $1.32 \times 10^{-5}$ & 2.00  \\
			\hline 1600 &  $8.22 \times 10^{-7}$ & 2.00  &  $3.30 \times 10^{-6}$ & 2.00  \\
			\hline 3200 &  $2.05 \times 10^{-7}$ & 2.00  &  $8.25 \times 10^{-7}$ & 2.00  \\
			\hline 6400 &  $5.14 \times 10^{-8}$ & 2.00  &  $2.06 \times 10^{-7}$ & 2.00  \\
			\hline
	\end{tabular}}
\end{table}

\subsection{Convergence towards a steady flow over a bump} \label{testcase3}

This test case aims to study the convergence towards a steady flow over a bump. It is a classical test for the shallow water equations adapted with the Coriolis source term in \cite{36bookZeitlinBouchut2007}. The topography is given by
$$z(x) = 
\begin{cases}
0.2 - 0.05(x-10)^2 & \text{ if } 8 < x < 12, \\
0 & \text{ otherwise.}
\end{cases}
$$
We consider the following initial data
$$h_0(x) = 0.33,\quad u_0(x) = 0.18/0.33 ,\quad v_0(x) = 0.$$

We compute the scheme on the domain $[0,25]$ with $N=200$ cells and we set $f = \frac{2 \pi}{50}$ and $g = 9.81$. The boundary conditions are set as
$$(hu)(x=0) = 0.18,\quad h(x=25) = 0.33,\quad v(x=0) = 0.$$ 

The numerical solution at time $T_{\max} = 200$  is represented in \cref{fig:SFOB_solution}. The time evolution of the steady state distance $\mathcal{E}_{\infty,j}^n$ is shown for both schemes in \cref{fig:SFOB_eps}. We can see these distances diminishing through time, which means both schemes actually converge towards a steady state.
%

\begin{figure}\label{fig:SFOB_solution}
	\centering
	\includegraphics[scale=0.65]{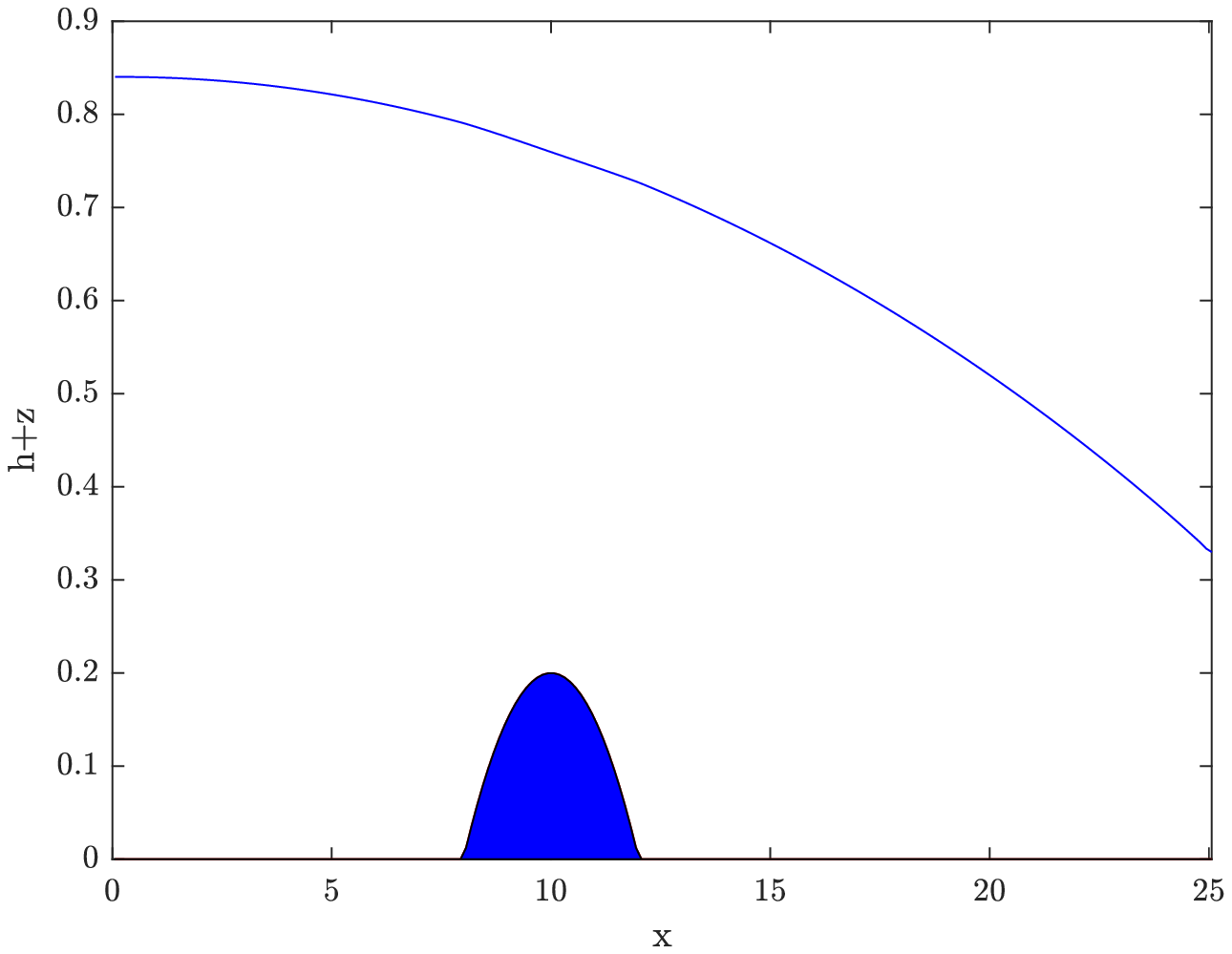}
	\includegraphics[scale=0.65]{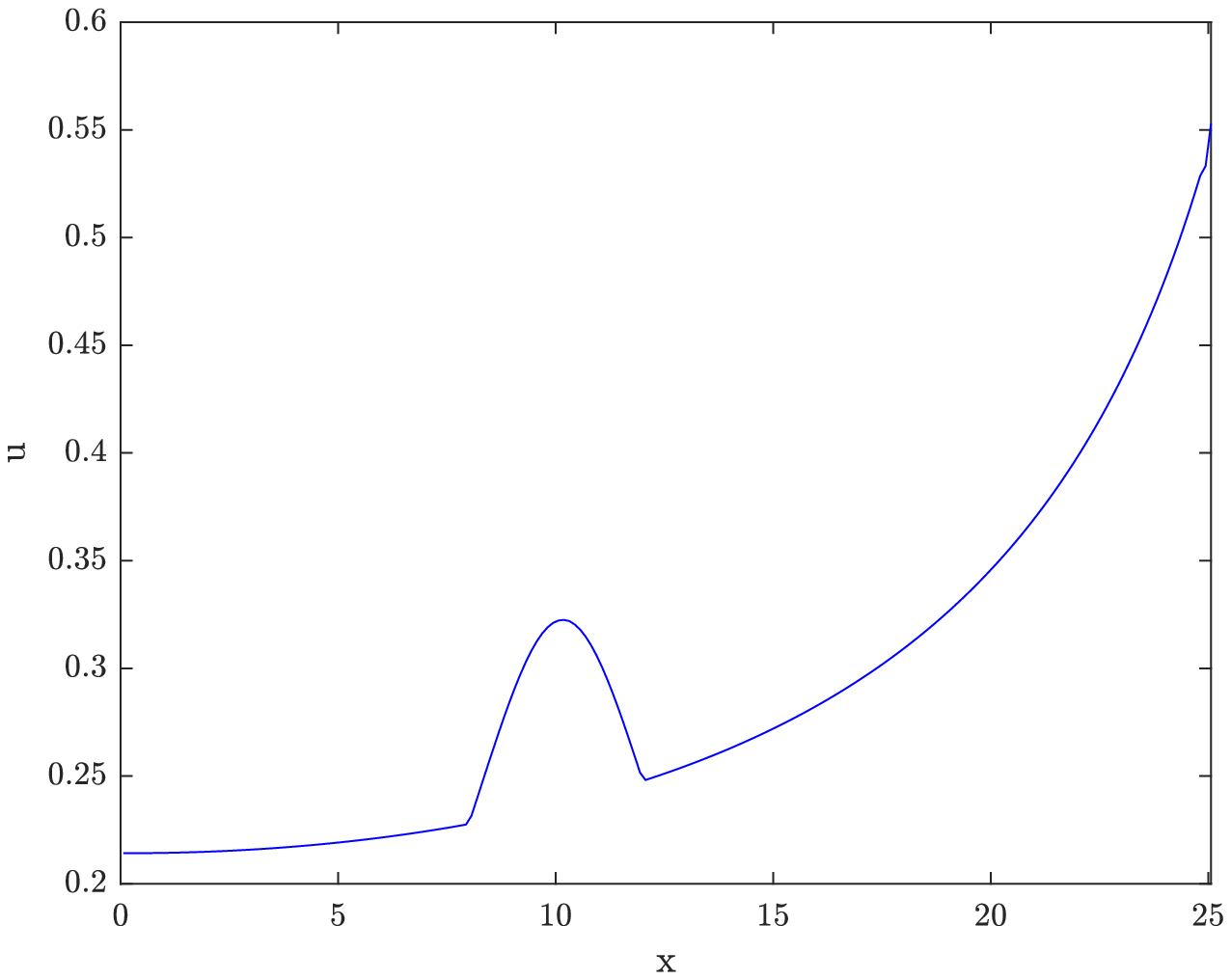}
	\includegraphics[scale=0.65]{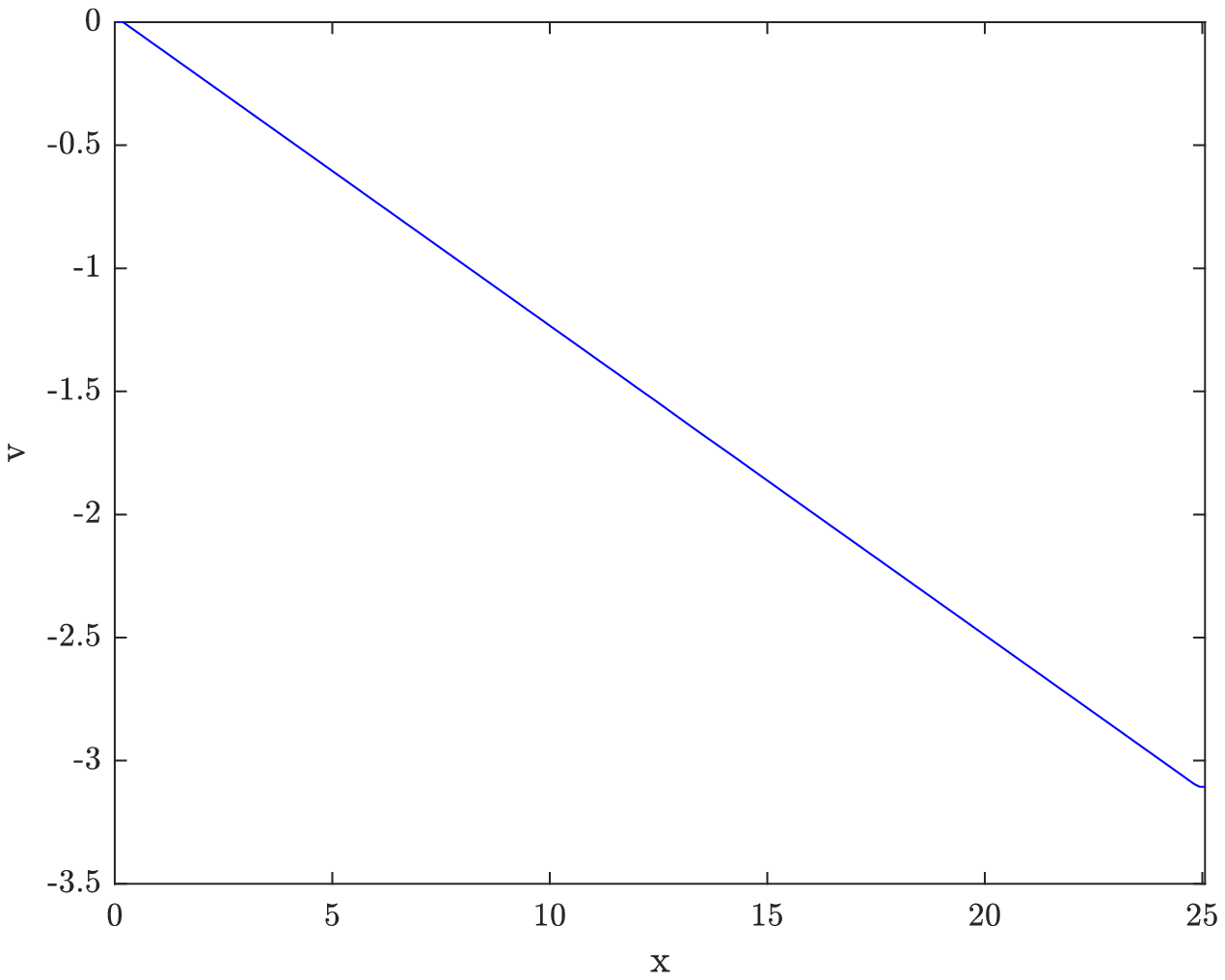}
	\caption{Approximate solution of the steady flow over a bump test case at time $T_{\max} = 200$}
\end{figure}

\begin{figure}\label{fig:SFOB_eps}
	\centering
	\includegraphics[scale=0.75]{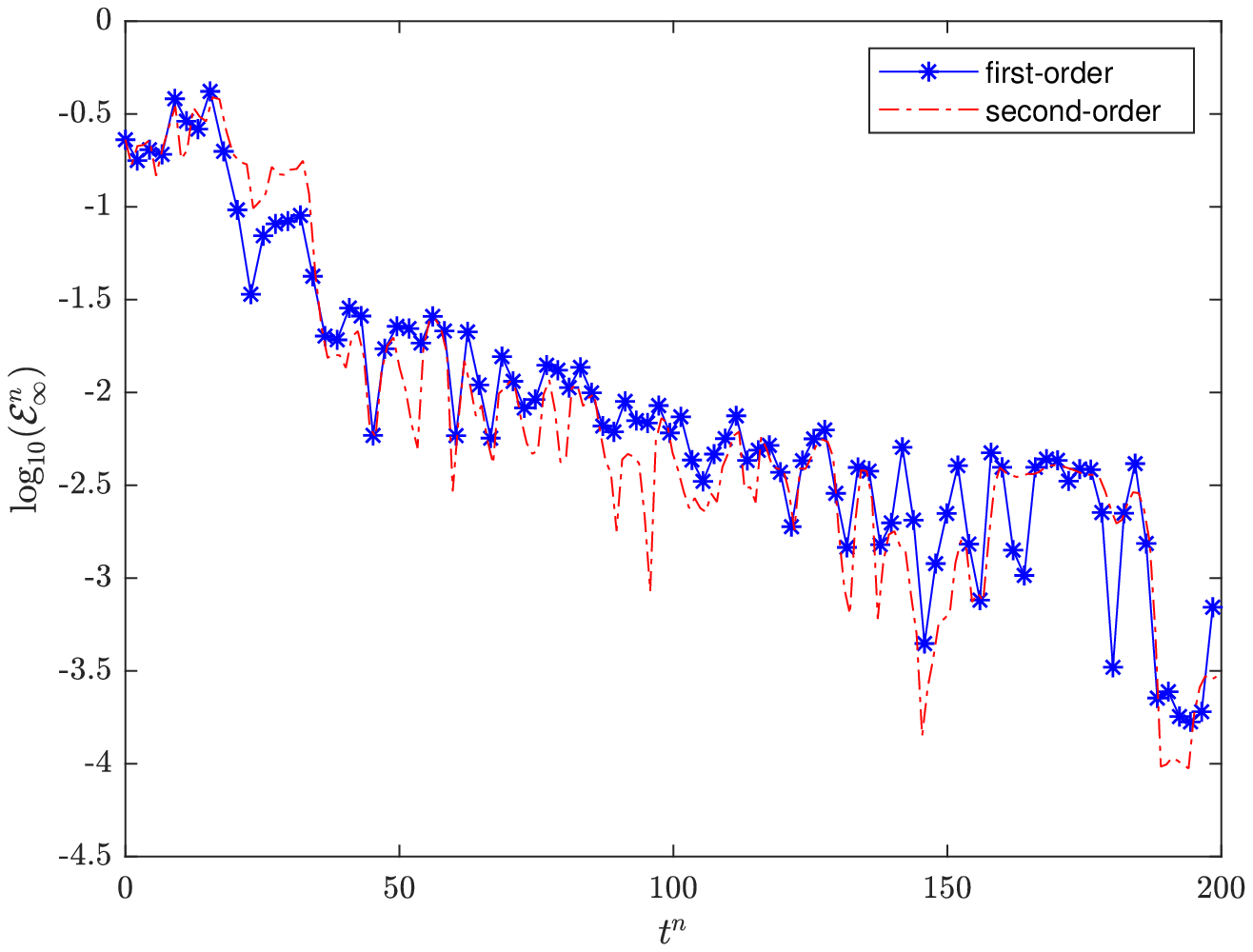}
	\caption{Steady flow over a bump, steady state distance $\mathcal{E}^n_\infty$ in logarithmic scale}
\end{figure}

\subsection{Stationary state in space} \label{testcase4}
This test case is based on a particular exact solution of the RSW equations without topography. For a constant initial condition $(h_0,u_0,v_0)$ fixed, the exact solution of RSW equations writes
$$\begin{aligned} 
& h(x,t) = h_0, \\
& u(t) = u_0 \cos(ft) + v_0 \sin(ft),\\
& v(t) = v_0 \cos(ft) - u_0 \sin(ft).
\end{aligned}$$ 
For any fixed time $t \geq 0$, the solution remains constant in space. We compute the scheme on domain $[0,1]$ until time $T_{\max} = 1$. We choose $$h_0=1,\quad u_0=1,\quad v_0 = 1$$ as initial data, with the parameters $f=g=1$ and we use periodic boundary conditions. 

The solution is well-captured by the scheme as one can see in \cref{fig:Stationary state in space}, where we represent $hu$ and $hv$ with respect to time. Since the exact solution is known and constant in space, we can check the scheme's accuracy in time. We introduce the discrete $L^1$ error in time between the exact solution $w_{ex}$ and the numerical approximation at point $x_i$
$$ E_i =  \sum_{n} (t^{n+1}-t^n) |w_{ex}(x_i,t^n)-w_i^n |. $$
Let us notice that the choice of the point $x_i$ is irrelevant. We recover the expected order of accuracy in time as one can see in error \cref{tab:CosSin_error}.

\begin{figure}\label{fig:Stationary state in space}
	\centering
	\includegraphics[scale=0.75]{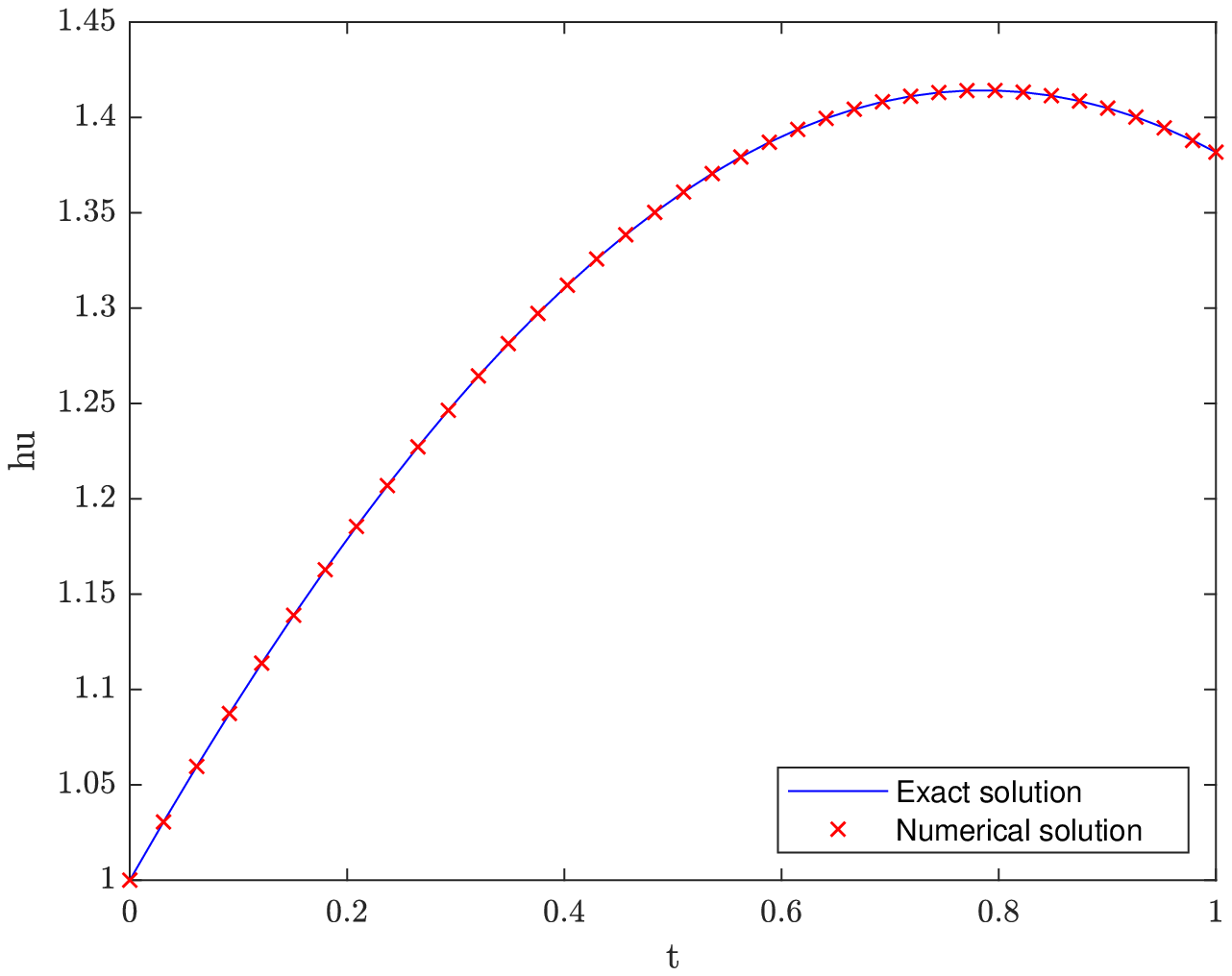}
	\includegraphics[scale=0.75]{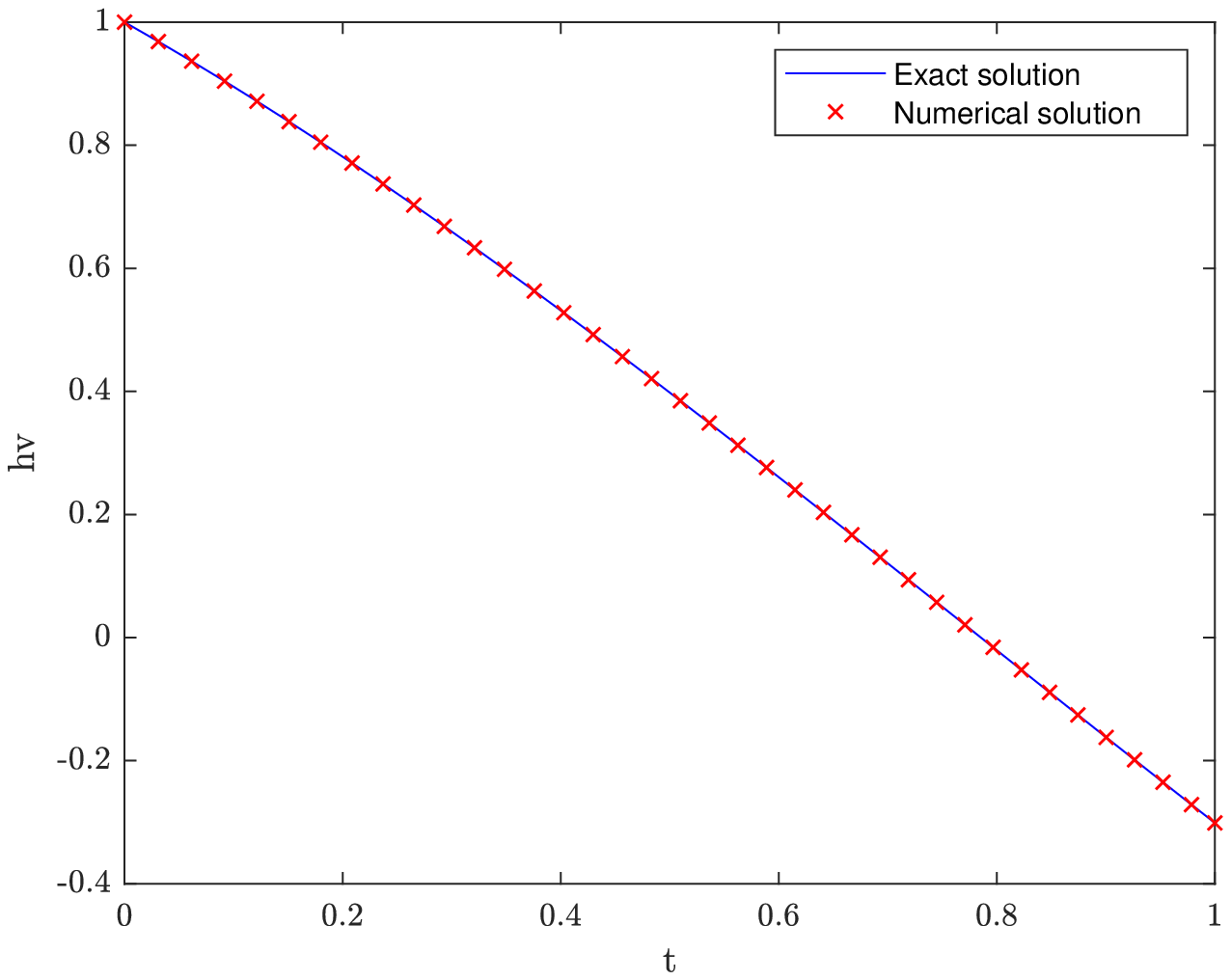}
	\caption{Stationary state in space at time $T_{\max} = 1$}
\end{figure}

\begin{table}{}\label{tab:CosSin_error}
	\centering
	\caption{$L^1$ error in time for the stationary in space test case at time $T_{\max} = 1$}
	
	\subtable[first-order scheme]{
		\begin{tabular}{|c|c|c|c|c|}
			\hline $N$    & \multicolumn{2}{c|}{hu} & \multicolumn{2}{c|}{hv}   \\
			\hline $200$  & $3.82 \times 10^{-4}$ & 0.99 & $8.06 \times 10^{-5}$  & 0.99 \\
			\hline $400$  & $1.91 \times 10^{-4}$ & 0.99 & $4.03 \times 10^{-5}$  & 0.99 \\
			\hline $800$  & $9.56 \times 10^{-5}$ & 0.99 & $2.01 \times 10^{-5}$  & 0.99 \\
			\hline $1600$ & $4.78 \times 10^{-5}$ & 0.99 & $1.01 \times 10^{-5}$ & 0.99 \\
			\hline $3200$ & $2.39 \times 10^{-5}$ & 0.99 & $5.04 \times 10^{-6}$ & 0.99 \\
			\hline $6400$ & $1.20 \times 10^{-5}$ & 0.99 & $2.52 \times 10^{-6}$ & 0.99 \\
			\hline
		\end{tabular}
	}
	
	\subtable[second-order scheme]{
		\begin{tabular}{|c|c|c|c|c|}
			\hline $N$    & \multicolumn{2}{c|}{hu} & \multicolumn{2}{c|}{hv}   \\
			\hline $200$  & $7.71 \times 10^{-9}$  & 1.99 & $3.58 \times 10^{-8}$  & 1.99 \\
			\hline $400$  & $1.92 \times 10^{-9}$  & 1.99 & $8.95 \times 10^{-9}$  & 1.99 \\
			\hline $800$  & $4.82 \times 10^{-10}$ & 2.00 & $2.24 \times 10^{-9}$  & 1.99 \\
			\hline $1600$ & $1.20 \times 10^{-10}$ & 2.00 & $5.60 \times 10^{-10}$ & 1.99 \\
			\hline $3200$ & $3.01 \times 10^{-11}$ & 2.00 & $1.40 \times 10^{-10}$ & 1.99 \\
			\hline $6400$ & $7.52 \times 10^{-12}$ & 2.00 & $3.50 \times 10^{-11}$ & 1.99 \\
			\hline
		\end{tabular}
	}
	
\end{table}

\section{Conclusions}\label{sec:conclusions}

In this work, we have built a second-order fully well-balanced scheme for the RSW system. In the first part, we have developed a fully well-balanced approximate Riemann solver by selecting carefully the numerical source term definitions and the relations used to define the intermediate states $w_L^\star$ and $w_R^\star$. The positivity of the variable $h$ has been recovered thanks to a cut-off procedure. We have proved in \cref{thm:first-order_scheme} that the resulting Godunov-type scheme satisfies all of the required features: consistency, positivity preserving and fully well-balanced property.

In the second part, we have proposed a way to extend the Godunov-type scheme to second-order. We have explained the limitations of the classical MUSCL method in view of the fully well-balanced property in the case of the RSW equations. Then we have adapted an idea proposed by \cite{46Berthon_Blood_flow,34DansacBerthonClainFoucher2016}, which consists in getting the standard MUSCL second-order scheme far from steady states and recovering the first-order fully well-balanced scheme near steady states. That procedure preserves the positivity of $h$ as proved in \cref{thm:second-order_scheme}. 

Finally, we have presented some numerical experiments to illustrate the robustness and the efficiency of both  first-order and second-order schemes.

This work can be easily extended to the two-dimensional RSW equations by involving a standard convex combination of 1D schemes by interface. 
Additionally, the Coriolis parameter has been assumed constant all along this paper.
It would be an interesting development of this work to consider a space-dependent Coriolis force, since it would be more realistic for large-scale simulations.



\section*{Acknowledgments}
The authors would like to thank C. Berthon and V. Michel-Dansac for the fruitful discussions.

\bibliographystyle{siamplain}
\bibliography{bib}

\end{document}